\documentclass[a4paper]{article}
\usepackage{geometry}
\usepackage{xcolor}
\usepackage{amssymb}
\usepackage{amsmath, amsthm, amsfonts}
\usepackage{graphicx}
\usepackage{lineno}
\usepackage{enumitem}
\usepackage{longtable}
\usepackage{environ}
\usepackage{bm}
\usepackage{thm-restate}
\usepackage{hyperref}
\usepackage{cleveref}

\usepackage[numbers]{natbib}
\usepackage{xurl}

\hypersetup{
	colorlinks=true, 
	linkcolor=cyan,
	citecolor=blue,
}

\newcommand{\re}[1]{\pm(#1)}

\newtheorem{theorem}{Theorem}[section]
\newtheorem{claim}[theorem]{Claim}

\newtheorem{lemma}[theorem]{Lemma}
\newtheorem{conjecture}[theorem]{Conjecture}
\newtheorem{proposition}[theorem]{Proposition}
\newtheorem{obs}[theorem]{Observation}
\theoremstyle{definition}

\newtheorem{definition}[theorem]{Definition}
\theoremstyle{remark}

\newcommand{\dist}{\operatorname{dist}}
\newcommand{\LL}{\mathcal{L}}
\newcommand{\NL}{\overline{\mathcal{L}}}

\newenvironment{myproof}{\noindent{\em Proof.}}{\hfill$\triangle$  \vspace{0.2cm}}

\newcommand{\im}{\operatorname{Im}}
\newcommand{\dom}{\operatorname{Dom}}

\newcommand{\ksf}{$\mathcal K_{2,s}$-free }
\newcommand{\ksff}{$\mathcal K_{2,s}$-free}

\title{Antidirected trees in dense digraphs}
\author{Maya Stein\thanks{Department of Mathematical Engineering and Center for Mathematical Modeling (CNRS IRL2807), University of Chile. Supported by FONDECYT Regular Grant 1221905,  by ANID Basal Grant CMM FB210005,  and by MSCA-RISE-2020-101007705 project {\it RandNET}. {\tt mstein@dim.uchile.cl}
} \and
	Ana Trujillo-Negrete\thanks{Center for Mathematical Modeling (CNRS IRL 2807), University of Chile, Santiago, Chile. Supported by ANID/Fondecyt Postdoctorado 3220838, and by ANID Basal Grant CMM FB210005. {\tt ltrujillo@dim.uchile.cl} }}

\date{ }

\begin{document}
	\maketitle

	\begin{abstract}
	We show that if $D$ is an $n$-vertex digraph  with  more than $(k-1)n$ arcs that does not contain any of  three forbidden digraphs, then $D$ contains every antidirected tree
	on $k$ arcs. The forbidden digraphs are 
	those orientations of $K_{2, \lceil k/12\rceil}$ 
	where each of the vertices in the class of size two has either out-degree $0$ or in-degree $0$.
	This proves a conjecture of Addario-Berry et al.~for a broad class of digraphs, and  generalises a result for $K_{2, \lfloor k/12\rfloor}$-free  graphs by Balasubramanian and Dobson. 
	
	We also show that every  digraph $D$   on $n$ vertices with more than $(k-1)n$ arcs  contains every antidirected $k$-arc caterpillar, thus solving   the above conjecture for caterpillars. This generalises a result of Perles.
\end{abstract}

 


\section{Introduction}

In 1963, Erd\H os and S\'os conjectured that every graph on $n$ vertices and with more than $(k-1)n/2$ edges contains, as a subgraph, each tree with $k$ edges~\cite{Erdos64}. This intriguing conjecture is still open in general, although it is known to hold for certain classes of host graphs characterised by forbidden subgraphs. The most general result in this direction is due to Balasubramanian and Dobson~\cite{Dob01} who showed that the Erd\H os-S\'os conjecture holds for $K_{2, \lfloor k/12\rfloor}$-free graphs. See~\cite{BPS3, rozhon}  for some more recent progress on the conjecture, and~\cite{survey} for a survey. 

Soon after the Erd\H os-S\'os conjecture appeared, the search for a corresponding statement for oriented trees in digraphs began. It quickly became evident that such a statement can only hold 
for 
{\it antidirected trees}, that is, oriented trees without directed paths of length two. To see this, take
a bipartite graph on sets~$A$ and~$B$, and orient every edge  from $A$ to $B$. This digraph has $|A||B|$ edges, and all its subgraphs are antidirected. 

According to~\cite{graham}, Erd\H os conjectured that  for every antidirected tree $T$ there is a constant $c_T$ such that every sufficiently large directed graph~$D$ on $n$ vertices and with at least $c_Tn$ arcs contains~$T$. This was confirmed by Graham~\cite{graham} in 
 1970. 
In 1982, Burr~\cite{burr} showed that one can take $c_T=4|E(T)|$ in Graham's result, namely, he showed that  every $n$-vertex digraph~$D$ with more than $4kn$ arcs contains each antidirected tree $T$ on $k$ arcs, although he did not believe the number $4kn $ to be optimal. Burr also provided an example where $(k-1)n$ arcs are not sufficient: Take the $k$-edge star with all edges directed outwards as the antidirected tree $T$, and let the host be the complete bipartite graph $K_{2k-2,2k-2}$ with half of the edges  at each vertex oriented in either direction. 

In 2013, 
Addario-Berry
et al.~\cite{ahlrt} formulated a conjecture which would imply that the digraph from Burr's example is extremal for containment of antidirected trees.

\begin{conjecture}[Addario-Berry, Havet, Linhares Sales, Reed and Thomassé 2013~\cite{ahlrt}]
	\label{conj:antidir}\label{antitree}\label{ESantitrees} 
 Every  digraph $D$ with more than $(k-1)|V(D)|$ arcs contains each  antidirected  tree with $k$ arcs. \samepage 
\end{conjecture}

 As the authors of~\cite{ahlrt}  observe, their conjecture implies the Erd\H os-S\'os conjecture, as any graph $G$ corresponds to a digraph having both arcs for each edge of $G$. 
 
Not much is known about Conjecture~\ref{ESantitrees}. In \cite{ahlrt}, it is shown to hold for all antidirected trees of diameter at most~3. Also, if
the digraph $D$ from  Conjecture~\ref{ESantitrees} is an {oriented}  graph, then 
 about $\frac 32 (k-1)n$ arcs are sufficient to guarantee all (of the up to two) antidirected paths with $k$ arcs appear as  a subgraph of $D$~\cite{antipaths}. This bound on the number of arcs has been improved very recently  to  about $\frac 43 (k-1)n$~\cite{chen2024}.
 Finally, Conjecture~\ref{ESantitrees} is asymptotically true in all large oriented graphs  $D$ for all large balanced antidirected trees of bounded maximum degree~\cite{camila}. 

\begin{figure}[t]
	\centering
	\includegraphics[width=0.78\textwidth]{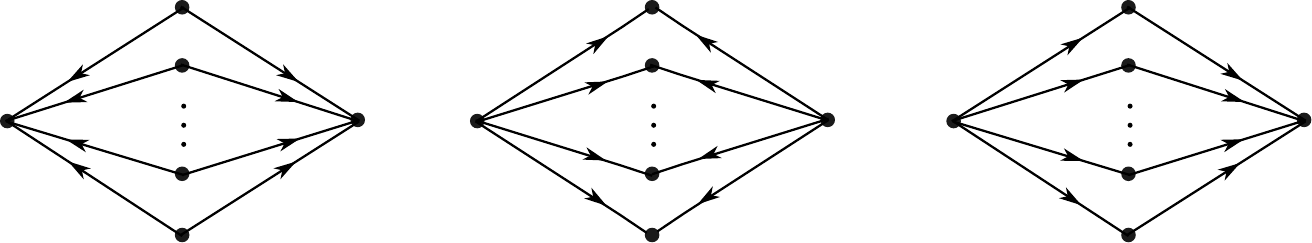}
	\caption{The three forbidden digraphs.}
	\label{fig:K_{2,s}}
\end{figure}

We show that Conjecture~\ref{ESantitrees} holds for a broad class of host digraphs, namely, for all digraphs that do not contain one of three fixed types of orientations of $K_{2, s}$, where $s=\lceil k/12\rceil$. 
These three orientations are depicted in Figure~\ref{fig:K_{2,s}}. They are the three orientations of $K_{2,s}$ that can be obtained by identifying the leaves of two antidirected stars having $s$ arcs each. 
We  say a digraph $D$ is {\it \ksf} if it does not contain any of  these three orientations of $K_{2,s}$.
With this definition at hand, our main result reads as follows.

\begin{theorem}
	\label{thm:main}
	For any $k\in\mathbb N$ and for $s=\lceil k/12\rceil$, every \ksf digraph $D$ with $n$ vertices and more than $(k-1)n$ arcs contains every antidirected tree
	with $k$ arcs. 
\end{theorem}

Note that Theorem~\ref{thm:main} is a natural digraph analogue of the resolution of the Erd\H os-S\'os conjecture for $K_{2, \lfloor k/12\rfloor}$-free graphs  from~\cite{Dob01} we mentioned at the beginning of the introduction. In fact, it is easy to see that  Theorem~\ref{thm:main} implies the result from~\cite{Dob01}.

We also prove  Conjecture~\ref{ESantitrees} for a special class of antidirected trees. A  {\it caterpillar} is a tree that consists of a  path and leaves attached to it.

\begin{proposition}
\label{prop:dirEScat}
	Every  digraph $D$ on $n$ vertices with more than $(k-1)n$ arcs  contains every antidirected caterpillar 
	with $k$ arcs. 
\end{proposition}

While Proposition~\ref{prop:dirEScat} is interesting on its own, it is also relevant for Theorem~\ref{thm:main}. Indeed, our proof of Theorem~\ref{thm:main} relies on
a variant of Proposition~\ref{prop:dirEScat}, namely, Proposition~\ref{prop:catmindeg} (which will be stated in Section~\ref{sec:cat}).
The proofs of both propositions  use the principal idea of an earlier proof of the Erd\H{o}s-Sós conjecture for caterpillars, due to Perles. His result was first mentioned in~\cite{moser}, without a proof. 
\begin{proposition}[Perles (see~\cite{moser})]\label{prop:EScat}
	Every graph $G$ on $n$ vertices with more than $(k-1)n/2$ edges  contains every caterpillar 
	with $k$ edges. 
\end{proposition}

We note that Proposition~\ref{prop:EScat} follows directly from Proposition~\ref{prop:dirEScat}, where the digraph~$D$ is obtained from the graph $G$ by replacing each edge  with   arcs in both directions.

Our paper is organised as follows. In Section~\ref{sec:overview}, we  state  Lemmas~\ref{lemma:onepointfive} and~\ref{lemma:third}, which together immediately imply Theorem~\ref{thm:main}. We also use  Section~\ref{sec:overview} to give an overview of the proofs of the two lemmas and Propositions~\ref{prop:dirEScat} 
 and~\ref{prop:catmindeg}.
 After introducing some notation in Section~\ref{sec:not}, we formally prove 
 the two propositions
  in Section~\ref{sec:cat}. 
  In particular, the proof of Proposition~\ref{prop:dirEScat} 
   can be found in Section~\ref{sec:dirEScat}, 
   and this section is entirely independent of the rest of the the paper. 
The only result from Section~\ref{sec:cat} that is  used in the rest of the paper
is Proposition~\ref{prop:catmindeg}, which  is shown  in Section~\ref{sec:minEScat}.

  The remainder of the paper is dedicated to the proof of Lemmas~\ref{lemma:onepointfive} and~\ref{lemma:third}. 
Section~\ref{sec:selection-subdigraph} is devoted to selecting a  subdigraph of the host digraph which obeys certain degree conditions, making it suitable to accommodate the antidirected tree. The results of Section~\ref{sec:selection-subdigraph}  will be useful in both 
Sections~\ref{sec:Delta2-small}~and~\ref{sec:lemma-third}, which contain the rest of the proofs of Lemmas~\ref{lemma:onepointfive} and~\ref{lemma:third}, respectively.  


\section{Overview}\label{sec:overview}
\subsection{Sketch of the proof of Propositions~\ref{prop:dirEScat} 
 and~\ref{prop:catmindeg}}

The idea of the proofs of  both propositions 
can be resumed as follows. We consider a drawing of the host 
 digraph in the plane, with all vertices on the unit circle, and arcs 
   being straight lines. (Arcs  
    are allowed to cross.) We inductively embed the caterpillar $T$, starting with an arc adjacent to one extreme of the spine of $T$, and then at each step adding a star, until we have embedded all of $T$. 

To be more precise, we consider a growing sequence $T_1, T_2, \ldots , T_r=T$ of subtrees of $T$, as in the previous paragraph.
The first subtree $T_1$ is an edge, and it could be embedded into any edge of the host. We then show inductively that for each of the subtrees $T_i$ of $T$ there are many   arcs~$e$ of the host digraph such that there is an embedding of $T_i$ into the host digraph, with $e_i$ embedded into $e$, and such that the embedding  has no crossing  arcs. The number of suitable  arcs will decrease while $i$ grows. However, one can show that in the last step, when $i=r$, there is still at least one   arc left, and therefore $T$ can be correctly embedded into the host.

\subsection{The proof of Theorem~\ref{thm:main}}

The proof of Theorem~\ref{thm:main} follows directly from the next two results, where $\Delta_2(T)$ is the value of the second highest total degree of a vertex of $T$. 
\begin{lemma}
	\label{lemma:onepointfive}
	Let $T$ be a $k$-arc antidirected tree, and let $D$ be a  \ksf digraph with $n$ vertices and more than $(k-1)n$ arcs, where $s=\lceil k/12\rceil$.
	If  $\Delta_2 (T)\le \lfloor k/4\rfloor +2$, 
	then $T$ embeds in $D$. 
\end{lemma}

\begin{lemma}
	\label{lemma:third}
	Let $T$ be  a $k$-arc antidirected tree, and let $D$ be a \ksf digraph with $n$ vertices and more than $(k-1)n$ arcs, where $s=\lceil k/12\rceil$.
	If $\Delta_2(T)\ge \lfloor k/4\rfloor +3$, then $T$ embeds in $D$.  
\end{lemma}

\subsection{Sketch of the proof of Lemmas~\ref{lemma:onepointfive} and~\ref{lemma:third}}

In both Lemmas~\ref{lemma:onepointfive} and~\ref{lemma:third}, we start by finding a subdigraph of $D$
that can serve as our host digraph. We discuss this step first, and then split our overview into the two cases stated in the lemmas. This corresponds to our division later into Sections~\ref{sec:selection-subdigraph},~\ref{sec:Delta2-small}~and~\ref{sec:lemma-third}.

\paragraph{Finding a suitable subdigraph $D'$ of $D$} 
Let us start by noting that if instead of a digraph~$D$ we were working with a graph $G$, then a condition of the density would directly imply the existence of  a subgraph of $G$ with at least the same density and obeying a minimal degree condition. Indeed, it is a very well known fact that every graph $G$ with more than $(k-1)|V(G)|/2$ edges has a subgraph~$H$ with more than $(k-1)|V(H)|/2$ edges that has minimum degree $\delta(H)\ge k/2$. One can find such a subgraph $H$ by successively deleting suitable vertices from $G$. This  property turns out to be very useful for finding subgraphs, as we can now limit our considerations to $H$, having obtained the minimum degree condition `for free'.

Unfortunately, an analogous statement  in digraphs does not hold, if we replace the minimum degree with the minimum semidegree (the minimum of all the out- and in-degrees). However, it has been observed~\cite{antipaths, maya-digraphs, camila}   that an analogous statement does hold if instead of using the minimum semidegree, we only require a condition on the {\it minimum pseudo-semidegree}. The minimum pseudo-semidegree of a non-empty digraph $D$ is defined as the maximum $d\in\mathbb N$ such that no vertex of $D$ has in- or out-degree in the interval $[1, d-1]$. The minimum pseudo-semidegree of an empty digraph~$D$ is $0$.   Using this notion, an analogue of the graph situation can be proved.  
\begin{lemma}[see~\cite{antipaths, maya-digraphs, camila}]\label{lempseu}
	Every digraph $D$ with more than $(k-1)|V(D)|$ arcs has a subgraph~$D'$ that has minimum pseudo-semidegree $\overline{\delta^0}(D')\ge  k/2$. 
\end{lemma}

The subdigraph $D'$ can be found by successively deleting arcs from vertices of too low in- or out-degree. For embedding antidigraphs, a lower bound on the minimum pseudo-semidegree seems to be similarly useful as a lower bound on the minimum semidegree (see~\cite{maya-digraphs} for a discussion). 
So we could now work in $D'$.

But, a second complication arises. In the situation described above for graphs, the subgraph $H$ of $G$ has the same average degree as $G$. Therefore, $H$ has a vertex of degree at least $k$. This could be crucial for the embedding of a tree with $k$ edges, as $T$ may have a vertex of degree exceeding $k/2$. Turning back to the digraph situation, if we wish to find an antidirected tree in a digraph $D'$, the property we need is that $D'$ has a vertex that is suitable for the vertex $v_1$ of highest total degree in $T$. If $v_1$ has high out-degree, then we need a vertex of high out-degree   in $D'$, and if $v_1$ has high in-degree, then we need a vertex of high in-degree   in $D'$.
Now, the subdigraph $D'$ from Lemma~\ref{lempseu} does not necessarily contain the vertex we need. It does have a vertex of out-degree at least $k$ or a vertex of in-degree at least $k$, but we cannot choose which of these two is present. So Lemma~\ref{lempseu} alone is not sufficient for our purposes.

In order to overcome this difficulty, we prove a new lemma that provides a different and more suitable subdigraph $D'$ of $D$. This is  Lemma~\ref{cor:subdigraph}. We believe that this lemma may be suitable for future use in other digraph problems.

The content of  Lemma~\ref{cor:subdigraph} is that in the above situation, $D$ has a subdigraph $D'$, which, in a certain way, maintains the high arc density of $D$, 
and in addition satisfies certain  conditions on the minimum in- and out-degrees of its vertices. The precise bounds on the degrees in $D'$ rely on a parameter $r$, which is part of the input of Lemma~\ref{cor:subdigraph}, and which, in the application of this lemma, will depend on $\Delta$, the maximum total degree of  $T$, and on other properties of $T$. 

Let us make this more precise. We can  and will assume the the vertex $v_1$ of highest total degree in $T$ has high out-degree (otherwise we can switch all orientations both in $T$ and in $D$ and resolve the problem there). So we need a vertex of high out-degree in $D$. Now, 
Lemma~\ref{cor:subdigraph} allows for two slightly different outcomes. That is, there are two sets of conditions of which the output subgraph $D'$ fulfills one (either conditions \ref{cor:1},~\ref{cor:2} and~\ref{cor-subdigraph-i}, or conditions~\ref{cor:1},~\ref{cor:2} and~\ref{cor-subdigraph-ii} of Lemma~\ref{cor:subdigraph}). 
For details, see the statement of Lemma~\ref{cor:subdigraph} in Section~\ref{sec:selection-subdigraph}.  

The main difference between~\ref{cor-subdigraph-i} and~\ref{cor-subdigraph-ii} is the following. In the former situation, we can guarantee a vertex of out-degree at least $k$, but the minimum pseudo-semidegree of $D'$ may fail to be at least $k/2$: it is bounded from below only by~$r$. In the second situation, the minimum pseudo-semidegree of $D'$ is at least $k/2$, but the largest out-degree is bounded from below only by $k-r$. To be able to work in either of these situations (recall that we do not know which is given to us by Lemma~\ref{cor:subdigraph}), we will have to make sure that in the application of  Lemma~\ref{cor:subdigraph} the parameter $r$ is chosen with extreme care, and  depending on $T$.

\paragraph{Proving Lemma~\ref{lemma:onepointfive}: the case when $\boldsymbol{\Delta_2\le \lfloor k/4\rfloor+2}$.} 
We distinguish two subcases: the maximum total degree $\Delta$ of $T$ is at most $\lfloor k/4\rfloor$, or it exceeds $\lfloor k/4\rfloor$. In the former case, it is sufficient to work in the subdigraph $D'$ given by Lemma~\ref{lempseu}, as all vertices in $T$ have low degrees. We first embed greedily a largest possible subtree $T'$ of $T$. If  $T'\neq T$, there is a vertex $w\in V(T')$ that still has an out- or in-neighbour $w'$ that needs to be embedded,  while all the out- or in-neighbourhood of its image $f(w)$ has been used. We unembed a suitable vertex of $T'$ (and all that comes `after it'), and embed $w'$ instead. We try to extend this embedding, and show that if we fail to extend it to all of $T$, then there is another vertex $z$ whose image $f(z)$ has large out- or in-degree into used parts of~$D'$. The existence of $w$ and $z$, together with the fact that $D'$ is a \ksf digraph, can be used to show that all other vertices do not have high degree into the part of $D'$ that has been used by both embeddings. (See Lemma~\ref{lemma:k/4neighbors} for a more precise statement of the property of \ksf digraphs that we use.) This allows us to modify our second embedding to an embedding of all of $T$. For all details, see Section~\ref{sec:lemma-first}.

Let us now turn to the case when $\Delta$ exceeds $\lfloor k/4\rfloor $. In this case we need to use  Lemma~\ref{cor:subdigraph} instead of Lemma~\ref{lempseu}. We then show that we can relatively easily embed $T$ unless the outcome of Lemma~\ref{cor:subdigraph} (applied with $r=\min\{\lceil k/2\rceil, k-\Delta\}$) is a subdigraph $D'$ of minimum pseudo-semidegree at least~$k/2$. That particular case is handled by Lemma~\ref{lem:case3b}, whose proof occupies most of  Section~\ref{sec:lemma-second}. In this proof, we first embed a subtree of~$T$ consisting of all vertices at distance at most two to a maximum degree vertex of~$T$. We then inductively extend this embedding to all of $T$, with our arguments relying heavily on  Lemma~\ref{lemma:k/4neighbors}, i.e.~the fact that $D$ is $\mathcal K_{2,s}$-free.  
For all further details, see Section~\ref{sec:lemma-second}.

\paragraph{Proving Lemma~\ref{lemma:third}: the case when $\boldsymbol{\Delta_2\ge \lfloor k/4\rfloor +3}$.}
We start by defining a subtree~$B_{uv}$ of $T$. The caterpillar $B_{uv}$ is a subdivided double-star that is defined as follows: If $u$ and $v$ are the two vertices of largest total degree in $T$, then we let $B_{uv}$ consist of  the path joining $u$ and $v$ plus all neighbours of $u$ and $v$ (see Definition~\ref{defdoublebroom}). 

Our plan is  to apply Lemma~\ref{cor:subdigraph} to find a subdigraph $D'$ of $D$. The input parameter $r$ for Lemma~\ref{cor:subdigraph}  will depend on  $T$, as we will specify next. 
It will be convenient to distinguish two cases, \ref{rem:1} and~\ref{rem:2}. One of the main characteristics of case \ref{rem:1} is that $B_{uv}$ is relatively small. The precise case distinction is somewhat technical and is made explicit	in Subsection~\ref{subs:step1}.
Now, in case~\ref{rem:1}, we choose the  input parameter $r$ for Lemma~\ref{cor:subdigraph} as $r=\lceil k/2\rceil$ and in case \ref{rem:2}  we choose $r=\min\{\lceil 5k/12\rceil,k-\Delta (T)\}$. Depending on the outcome of   Lemma~\ref{cor:subdigraph}, we further subdivide case~\ref{rem:2} into cases \ref{rem:2.1} and~\ref{rem:2.2}.  

Once $D'$ is found, our embedding process consists of two  steps: 
first, we embed $B_{uv}$, and second, we extend this to an embedding of all of $T$. In cases \ref{rem:1} and \ref{rem:2.1}, the tree $T$ will be embedded completely into $D'$, while in case \ref{rem:2.2}, some vertices of $D-D'$ will be used for certain leaves of~$T$.

The first step, embedding $B_{uv}$, is executed in Section~\ref{subs:double-broom}. 
We make use of several approaches that exploit different properties of  $B_{uv}$. 
For instance, if $|B_{uv}|\le 3k/4$ it is easy to use the fact that $D$ is \ksf to see that $B_{uv}$ can be embedded greedily into $D'$. We treat all of case~\ref{rem:1} similarly.  In case \ref{rem:2.1}, we use Proposition~\ref{prop:catmindeg} to embed $B_{uv}$ in $D'$. Case \ref{rem:2.2} is much more complicated. We start by taking a maximal subdigraph $T'$ of $B_{uv}$
that embeds into $D$, where we require that all but certain leaves of $T$ are embedded in $D'$. We show that the closed neighbourhood of either $u$ or $v$ belongs to $T'$. Then, heavily using the fact that $D$ is a \ksf digraph, we prove that $T'$ contains more than $3k/4$ vertices, and use this to show that $T'=B_{uv}$, as desired. For this, we carefully analyse different cases that depend on the difference of the degrees of $u$ and $v$, on the number of leaves of $T$ adjacent to $u$ or $v$, and on whether $v$ is an out- or an in-vertex.

In the second step, which  is accomplished in Section~\ref{subs:extension}, we extend the embedding of $B_{uv}$ to an embedding of all of $T$. Again, we will rely on the property that $D$ is  $\mathcal K_{2,s}$-free.  It turns out that it is easy to embed $T$ unless we are in case \ref{rem:2.1} and $D'$ has minimum pseudo-semidegree at least  $r=\lceil 5k/12\rceil$. If we fail to embed $T$ in this case, there is vertex $w\in V(T)$ with an unembedded neighbour $w'$ such that all neighbours of the image $f(w)$
of $w$ are used. One of them is used by a non-neighbour of $w$, which can be swapped for $w'$, while keeping an essential part of the original embedding. We try to extend the new embedding: if we fail there is another vertex $z$ with properties as $w$ above, and we find a contradiction to $D$ being  $\mathcal K_{2,s}$-free.


\section{Notation}\label{sec:not}

Given a graph $G$, we write $|G|=|V(G)|$ and $e(G)=|E(G)|$ for the order and size of $G$, respectively. 
The minimum, average and maximum degree of $G$ are denoted by $\delta(G)$, $d(G)$ and $\Delta(G)$, respectively. 
For a digraph $D$ we use similar notation: $|D|=|V(D)|$ and $a(D)=|A(D)|$ are the order and size of $D$. 
When referring to a vertex $a$ of $D$, we may write $a\in V(D)$ or simply $a\in D$. 
A digraph $D$ is {\it antidirected} if it contains no directed path of length two. 
A {\it double-star} is a digraph obtained from two directed stars by adding an arc between their centers.  

Given a vertex $a$ of a digraph $D$, let $N^+(a)$ and $N^-(a)$  denote its (out- and in- )neighborhood, respectively. The closed (out- or in-)neighborhood is denoted as $N^+[a]$ and $N^-[a]$, respectively. 
Regarding minimum and maximum degrees we use the following notation: 
\begin{align*}
\delta^+(D)&:=\min_{a\in D}\{\deg^+(a)\},  & \delta^-(D)&:=\min_{a\in D}\{\deg^-(a)\}, \\
\Delta^+(D)&:=\max_{a\in D}\{\deg^+(a)\},  & \Delta^-(D)&:=\max_{a\in D}\{\deg^-(a)\}. 
\end{align*}
Additionally, we define the \textit{minimum semidegree} of $D$ as $\delta^0(D):=\min \{\delta^+(D),\delta^-(D)\}$. 
We also define the \textit{minimum pseudo-out-degree} $\overline{\delta^+}$ as follows: if $D$ has no edges, we let
$\overline{\delta^+}(D)=0$, otherwise $\overline{\delta^+}(D)$ is the minimum  $d\in \mathbb{N}$
such that  for all vertices $a\in D$ we have $\deg^+(a)=0$ or $\deg^+(a)\ge d$. The \textit{minimum pseudo-in-degree}
$\overline{\delta^-}(D)$ is defined similarly. Furthermore, the \textit{minimum pseudo-semidegree} 
of $D$ is $\overline{\delta^0}(D):=\min\{\overline{\delta^+}(D),\overline{\delta^-}(D)\}$.

Given a digraph $D$, we are typically interested in its vertices with positive in-degree or out-degree.
We denote these sets of vertices as:   
\[D^+:=\{a\in D:\deg^+(a)>0\} \text{\quad and \quad} D^-:=\{a\in D:\deg^-(a)>0\}. \]

Note that if $D$ has no isolated vertices, then $|V(D)|\ge |D^+|+|D^-|$, with equality when~$D$ is  antidirected. 
A non-isolated vertex $a$ of an antidirected graph $D$ is an \textit{out-vertex} if it is in $D^+$ and an \textit{in-vertex} otherwise.   
We define $\re{a}\in \{+,-\}$, the \textit{sign} of $a$, as $+$ if $a$ is an out-vertex and as $-$ otherwise. 
A {\it rooted antidirected tree} is an antidirected tree with a designated vertex as its root. 
Given a rooted antidirected tree $T$ with root $w$, for a vertex $x\in T\setminus \{w\}$, the \textit{parent} of $x$,
denoted as $p_x$, is the (unique) vertex of $N^{\re{x}}(x)$ that belongs to the path joining $w$ and $x$. A vertex $y$ is a {\it descendant} of $x$ if $x$ is on the path joining $w$ and $y$.

Given an antidirected tree $T$ and a digraph $D$, an {\it embedding} from $T$ to $D$ is an injective 
function $f: V(T) \to V(D)$ preserving adjacencies, that is, for each arc $uv \in T$, $f(u)f(v)$ is an arc of $D$.
If such an embedding exists, we say that $T$ \textit{embeds} in $D$.   
The domain and image of the embedding $f$ are denoted by $\dom f$ and $\im f$, respectively. 
For a subset $X\subseteq V(T)$ and an embedding $f$ from $T$ to $D$, we write $f(X)$ to denote
the image of $X$ in $D$. Note that an embedding $f$ from $T$
to~$D$ induces an injective function $\varepsilon_f$ mapping  arcs to arcs. 

Notation that is only needed in one of the sections of the paper will be introduced there.


\section{Caterpillars}
\label{sec:caterpillars}\label{sec:cat}

In Section~\ref{sec:dirEScat}, we adapt Perles' proof (see~\cite{blog, moser}) of Proposition~\ref{prop:EScat}  to a proof of Proposition~\ref{prop:dirEScat}, that is,  we prove Conjecture~\ref{ESantitrees} for antidirected caterpillars. 

In Section~\ref{subs:cat-mindeg-digraphs}, we prove a variant  
of Proposition~\ref{prop:dirEScat}, namely, Proposition~\ref{prop:catmindeg}. For  the proof of this proposition, it will be a good warm-up exercise to read Section~\ref{sec:dirEScat} first, although   this is not strictly necessary for understanding the rest of the paper.

\subsection{Antidirected caterpillars in dense digraphs}\label{sec:dirEScat}

In this subsection, we prove Proposition~\ref{prop:dirEScat}.
Let us define the necessary notation in the digraph setting.  
A {\it convex digraph} $D$ is a drawing of a directed graph in the plane, with the vertices lying on the unit  circle and the arcs being straight segments (with a direction)
between the corresponding vertices. So, if for some $x,y\in V(D)$ both arcs $xy$ and $yx$ exist, they fully overlap each other.
Every arc $xy$ or $yx$ of $D$ divides $V(D)\setminus\{x,y\}$ into two sets,   $V^{x,y}_{\leftarrow x}$, which contains all vertices that lie after $x$ and before $y$ on the circle, counting clockwise, 
and $V^{x,y}_{x\rightarrow}$, which contains all other vertices of $V(D)\setminus\{x,y\}$.

We can view any (undirected)
 caterpillar as a tree $T_{un}$\footnote{The subscript ``$un$'' indicates that the tree  is undirected} that consists of a path, which we call its 
{\it spine}, and  leaves   attached to the inner vertices of the spine. So the first and last vertex of the spine are leaves of $T_{un}$. The  last vertex of the spine is called the {\it final vertex} of the caterpillar, and the edge incident with it is the {\it final edge}.

Let $T_{un}$ be a caterpillar with final vertex $u$, final edge $vu$, and spine $P_{un}$.  Let $D$ be a convex digraph. We say an arc  $a$ is 
\textit{good} for an antidirected orientation $T$ of $T_{un}$ if 
there is an embedding $f$
of $T$ in $D$ such that the following holds. If $x=f(u)$ and $y=f(v)$, then
\begin{itemize}
	\item  the arc of $T$ that contains $\{u,v\}$ is mapped by $\varepsilon_f$ to $a$, 
	\item if $|V(P)|$ is odd    then $V^{x,y}_{\leftarrow x}$ contains no images of  vertices of $V(T)$, and 
	\item if $|V(P)|$ is even  then $V^{x,y}_{x\rightarrow}$ contains no images of  vertices of $V(T)$.
\end{itemize}

\begin{lemma}
	\label{lemma:good-middle-arc}
	Let $k, n\in\mathbb N^+$, let $T$ be an antidirected caterpillar with $k$ arcs and let $D$ be a convex   
	digraph of order~$n$. 
	Then $D$ contains at least $a(D)-(k-1)n$ good arcs for $T$.  
\end{lemma}

Before we prove Lemma~\ref{lemma:good-middle-arc}, let us see how it implies Proposition~\ref{prop:dirEScat}.

\begin{proof}[Proof of Proposition~\ref{prop:dirEScat}]
Draw $D$ as a convex  digraph  $D'$  in the plane. As  $D'$ has more than $(k-1)n$ arcs, it contains at least one good arc for $T$ by Lemma~\ref{lemma:good-middle-arc}, and thus $T$ embeds in $D$.
\end{proof}

\begin{proof}[Proof of Lemma~\ref{lemma:good-middle-arc}]
	We proceed by (strong) induction on $k$. 
	In the base case $k=1$, we have $a(T)=1$, and every arc of $D$ is good for $T$. 
	
	For the induction step, let $D$ and $T$ with spine $P$ be given. Note that $|V(P)|\ge 3$, let $w, v, u$ be the last three vertices on $P$. 
	We assume that $|V(P)|$ is even; 
	the other case is analogous (each occurrence of $V^{x,y}_{\leftarrow x}$ would become $V^{x,y}_{\rightarrow x}$ and vice versa). We also assume that the arc between $u$ and $v$ is directed from $v$ to $u$, as the other case is analogous (each occurrence of the prefix {\it out-} would be replaced with {\it in-} and vice versa).
	
	Let $T'$ be  obtained from $T$ by deleting all   out-neighbours of $v$  except~$w$, and let $m$ be the number of vertices we have deleted. Since $T'$ has $k-m$ arcs, 
	by the induction hypothesis, $D$ has at least $a(D)-(k-m-1)n$ good arcs for $T'$.  
	
	For each out-vertex $x\in D$ we   order $V(D)\setminus \{x\}$ clockwise starting with the vertex immediately
	after $x$ on the circle. This induces a total order $\tau_x$ on the outgoing arcs from $x$. We call an arc $xy$ \textit{nasty}
	if it is one of the first $m$ outgoing arcs from $x$. 
	Let $F$ be the set of arcs of $D$ that are good for $T'$ and not nasty. So, 
	\[|F|\ge a(D)-(k-m-1)n-mn=a(D)-(k-1)n.\]
	
	We now define an injection, $\varphi$, from $F$ to $A(D)$ by setting $\varphi(xy):=xz$ for $xy\in F$, where $xz$ is $m$ positions before  $xy$ in $\tau_x$.  We will show that each arc from $\varphi(F)$ is good for $T$, which in turn completes the proof. 
	
	For this, consider $xy\in F$, and an embedding of $T'$ mapping $vw$ to $xy$ and having the property that 
	$V^{x,y}_{\leftarrow x}$ contains no images of  vertices of $V(T')$.
	(Such an embedding exists because $xy$ is a good arc for $T'$.) 
	We extend this embedding by embedding $u$ in $z$, and   embedding all vertices of $V(T')\setminus V(T)$ into 
	vertices $z'$ satisfying that $z'$ is an out-neighbour of $x$ and lies between 
	$x$ and $z$ on the circle. The choice of $z$ implies the existence of a sufficient number
	of such vertices $z'$. It is readily seen that $V^{xz}_{z\rightarrow}$ contains no vertex of $\varphi(T)$, which is as desired.
\end{proof}

\subsection{Antidirected caterpillars in other compatible digraphs}
	\label{sec:minEScat}
\label{subs:cat-mindeg-digraphs}
In this subsection we state and prove Proposition~\ref{prop:catmindeg}, which will be used later on in the paper.

\begin{proposition}\label{prop:catmindeg}
Let $n,k\in\mathbb N^+$, and let $D$ be a digraph on $n$ vertices with $a(D)>\frac{1}{2}(k-1)(|D^+|+|D^-|)$. Let $T$ be an antidirected caterpillar  with $a(T)=k$. Suppose at least one of the following holds:
\begin{itemize}
\item $|D^+|\le |D^-|$ and $|T^+|\le |T^-|$; or
\item $|D^+|\ge |D^-|$ and $|T^+|\ge |T^-|$.
\end{itemize}
Then $T$ embeds in $D$.
\end{proposition}

For the proof of Proposition~\ref{prop:catmindeg}, we need the following lemma.
We will continue to use the notation for antidirected caterpillars and convex digraphs from the previous subsection. 

\begin{lemma}
	\label{lemma:good-middle-arc-mindeg}
	Let $n\in\mathbb N^+$, let $D$ be a convex   
	digraph of order~$n$. Let $T$ be an antidirected caterpillar. 
	Then $D$ contains at least $a(D)-(|T^+|-1)|D^-|-(|T^-|-1)|D^+|$ good arcs for~$T$. 
	\end{lemma}
	
	\begin{proof}[Proof of Proposition~\ref{prop:catmindeg}]
	We assume that $|D^+|\le |D^-|$ and $|T^+|\le |T^-|$, the other case is analogous. 
	If we can apply Lemma~\ref{lemma:good-middle-arc-mindeg} to a convex drawing of $D$, we are done. So all we need to see is that 
	\begin{equation*}
	\frac{k-1}{2}(|D^+|+|D^-|)\ \ge \ (|T^+|-1)|D^-|+(|T^-|-1)|D^+|,
	\end{equation*}
	or equivalently,
	\begin{equation}\label{allworksout}
	\left(\frac{k+1}{2} - |T^+| \right)|D^-| \ \ge \ \left( |T^-| - \frac{k+1}{2} \right)|D^+|.
	\end{equation}
	In order to see~\cref{allworksout}, let us first observe that by the conditions of the proposition and given that $|T^+|+|T^-|=k+1$, we have that $|T^+|\le (k+1)/2\le |T^-|$. Thus the terms $(k+1)/2 - |T^+|$ and $|T^-| - (k+1)/2$ are  non-negative. So~\cref{allworksout} follows directly from the fact that $|D^+|\le |D^-|$, which holds by the assumptions of the proposition.
	\end{proof}

\begin{proof}[Proof of Lemma~\ref{lemma:good-middle-arc-mindeg}]
	We proceed by (strong) induction on $a(T)$, noting that for $a(T)=1$, there is only one arc in~$T$ and thus $|T^+|=|T^-|=1$, 
	and every arc of $D$ is good for $T$. 
	For the inductive step, note that~$T$ has at least three vertices, and let $w, v, u$ be the last three vertices on its spine. 
	
	From now on, we will assume that the arc between $u$ and $v$ is directed from $v$ to $u$, as the other case is analogous (more precisely, each occurrence of the prefix {\it out-} would be replaced with {\it in-} and vice versa, all signs $+$ and $-$ would change, as well as directions on the circle).
	
	Obtain $T'$ from $T$ by deleting all  out-neighbours of $v$  except~$w$. Let $m$ be the number of vertices we have deleted, and note that $|(T')^+|=|T^+|$ and $|(T')^-|=|T^-|-m$. So,
	by the induction hypothesis, we know that $D$ has at least $$a(D)-\big(|T^+|-1\big)|D^-|-\big(|T^-|-m-1\big)|D^+|$$ good arcs for $T'$.  
	
	As in the previous proof, we order, for each out-vertex $x\in D^+$, the set $V(D)\setminus \{x\}$ clockwise starting with the vertex immediately
	after $x$ on the circle. This induces a total order $\tau_x$ on the outgoing arcs from $x$. We call an arc $xy$ \textit{nasty}
	if it is one of the first $m$ outgoing arcs from~$x$. 
	
	Let $F$ be the set of arcs of $D$ that are good for $T'$ and not nasty for some $x\in D^+$. There are at most $m|D^+|$ such arcs. So, 
	\begin{align*}
		|F|&\ge a(D)-\big(|T^+|-1\big)|D^-|-\big(|T^-|-m-1\big)|D^+|  -m|D^+|\\&= a(D)-\big(|T^+|-1 \big)|D^-|-\big(|T^-|-1\big)|D^+|.
	\end{align*}

	Define an injection $\varphi$ from $F$ to $A(D)$ by setting $\varphi(xy):=xz$ for $xy\in F$, where $xz$ is $m$ positions before  $xy$ in $\tau_x$.  As in the previous proof, one can verify that each arc from $\varphi(F)$ is good for $T$, which completes the proof. 
\end{proof}


\section{Selecting a subdigraph for the tree embedding} ~\label{sec:selection-subdigraph}
As explained in Section~\ref{sec:overview}, given a digraph $D$ satisfying the conditions of the Conjecture~\ref{antitree}, our aim is to identify a subdigraph $D'$ of $D$ that 
obeys certain conditions on the degrees of its vertices. This is made explicit in Lemma~\ref{cor:subdigraph} below, which is the main result of this section. 

In order to prove  Lemma~\ref{cor:subdigraph}, we will use the digraph $D$ to construct a bipartite graph $G$, find a suitable subgraph $H$ of $G$, and then translate $H$ back to a subdigraph $D'$ of $D$. Lemma~\ref{lemma:bipartite} below takes care of finding the subgraph $H$ in $G$. For its proof, the following observation will come in handy. 

\begin{obs}
	\label{obs:deleting}
	Let $G$ be an $n$-vertex graph with $e(G)>(k-1)n/2$, and let $H$ be obtained from $G$ by deleting either 
	a vertex $u\in G$ with $\deg(u)<k/2$ or a pair of vertices $u,v\in G$ with $\deg(u)+\deg(v)<k$. 
	Then $e(H)>(k-1)|H|/2$.  
\end{obs}
\begin{proof}
If we deleted one vertex $u$ as above, then
\[e(H)=e(G)-\deg(u)>\frac{1}{2}(k-1)|G|-\frac{k-1}{2}=\frac{1}{2}(k-1)|H|,\]
and if we deleted two vertices $u$ and $v$ as above, then
\[e(H)\ge e(G)-(\deg(u)+\deg(v))>\frac{1}{2}(k-1)|G|-(k-1)=\frac{1}{2}(k-1)|H|.\]
\end{proof}

Given a graph $G$ and a set $X\subseteq V(G)$, we define $\delta_{G}(X):=\min\limits_{x\in X}\{\deg_G(x)\}$. We now state and prove Lemma~\ref{lemma:bipartite}. 

\begin{lemma}
	\label{lemma:bipartite}
	Let $k$ and $r$ be positive integers with $r\le \lceil k/2\rceil$.  
	Let $H=(A,B)$ be a bipartite graph with $|A|=|B|$ and $e(H)>(k-1)|H|/2$. Then there exists a subgraph 
	$H'=(A',B')$ of $H$ satisfying the following: 
	\begin{enumerate}[label=(\arabic*)]
		\item \label{item:1} $e(H')>(k-1)|H'|/2$; 
		\item  \label{item:2} for each $a\in A'$ and $b\in B'$ we have $\deg_{H'}(a)+\deg_{H'}(b)\ge k$. 
	\end{enumerate}
	Moreover, $H'$ satisfies one of the two following conditions: \setcounter{enumi}{2}
		\begin{enumerate}[label=(3-\Roman*)] 
			\item \label{bip-i}$\delta_{H'}(A')\ge k/2$, $\delta_{H'}(B')\ge r$, 
			there is $a\in A'$ with $\deg_{H'}(a)\ge k$, and  $|A'|\le |B'|$; or,  
			\item \label{bip-ii}$\delta(H')\ge k/2$, there is $b\in B'$ with $\deg_{H'}(b)\ge k$, 
			and for each $a\in A'$ we have $\deg_{H}(a)>k-r$. 
		\end{enumerate}
\end{lemma}
\begin{proof}
	We start by iteratively deleting some vertices of $H$ according to the following rules. First, if $\deg(u)< k/2$ for some vertex $u\in A$, we delete $u$. Second, if $\deg(u)+\deg(v)<k$
	for some vertices $u\in A$ and $v\in B$, we delete both. We stop once none of the rules can be applied anymore. Call the obtained graph  $H_1=(A_1,B_1)$.  Note that $H_1$ is a non-empty graph by Observation~\ref{obs:deleting}. By construction,
	\begin{enumerate}[label=(\roman*)]
	\item \label{eins} $\deg_{H_1}(a)\ge k/2$ for each $a\in A_1$, 
	\item \label{zwei}
	$\deg_{H_1}(a)+\deg_{H_1}(b)\ge k$ for each $a\in A_1$, $b\in B_1$, and
	\item \label{drei} $|A_1|\le |B_1|$. 
		\end{enumerate}  
	Furthermore,  it is easy to calculate that 
	\begin{enumerate}[label=(\roman*)]\setcounter{enumi}{3}
	\item \label{vier} $e(H_1)>(k-1)|H_1|/2$.  
		\end{enumerate}
	
	Observe that if  $\deg_{H_1}(b)\ge r$ for each vertex $b\in B_1$, then we can take $H':=H_1$. Indeed, conditions  \ref{item:1} and \ref{item:2} hold by~\ref{vier} and~\ref{zwei}, and $\delta_{H_1}(A_1)\ge k/2$ and $|A_1|\le |B_1|$ by~\ref{eins} and~\ref{drei}. 
	Moreover, because of~\ref{drei}, $|A_1|\le |H_1|/2$ and thus we can use~\ref{vier}
		to see that there is a vertex $a^*\in A_1$ with $\deg_{H_1}(a^*)\ge k$. 
	
	 So, from now on, we can  suppose that $\deg_{H_1}(b_1)<r$ for some vertex $b_1\in B_1$.  By~\ref{zwei}, we have that $\deg_{H_1}(a)\ge k-\deg_{H_1}(b_1)>k-r$ for each  $a\in A_1$.
	 We iteratively delete vertices from $H_1$ whose degree is less than $k/2$ until there are no such vertices. 
		Let $H'=(A',B')$ denote the obtained subgraph of $H_1$. 
		By construction,  for each $a\in A'$ we have $\deg_{H}(a)\ge \deg_{H_1}(a)> k-r$. Also by construction, we know that $\delta(H')\ge k/2$. 
		In particular, \ref{item:2} holds, and it is easy to calculate that \ref{item:1} holds. 
		So,  if $|A'|>|B'|$ then there is a vertex $b\in B'$ with $\deg_{H'}(b)\ge k$ and thus \ref{bip-ii} holds.  
		On the other hand, 
		if $|A'|\le |B'|$, then there is $a\in A'$ with $\deg_{H'}(a)\ge k$, 
		and so \ref{bip-i} holds. 
		\end{proof}

We are now ready to state and prove the main result of this section:
\begin{lemma}
	\label{cor:subdigraph}
	Let $k$, $r$ and $n$ be positive integers with $r\le \lceil k/2\rceil$ and $k\le n$. Let $D$ be a digraph of order $n$ and size $a(D)>(k-1)n$. 
	Then $D$ has a subdigraph $D'$ such that
		\begin{enumerate}[label=(\arabic*)]  
		\item \label{cor:1} $a(D')>(k-1)(|(D')^+|+|(D')^-|)/2$; 
		\item \label{cor:2} $\deg_{D'}^+(a)+\deg_{D'}^-(b)\ge k$ for each $a\in (D')^+$ and $b\in (D')^-$. 
	\end{enumerate}
	Additionally, $D'$ satisfies one of the following conditions: 
	\begin{enumerate}[label=(3-\Roman*)]
		\item \label{cor-subdigraph-i} 
		$\overline{\delta^+}(D')\ge k/2$, $\overline{\delta^-}(D')\ge r$, 
		there is $a\in V(D')$ with $\deg^+_{D'}(a)\ge k$, and  
		$|(D')^+|\le |(D')^-|$; or, 
		\item \label{cor-subdigraph-ii} 
		$\overline{\delta^0}(D')\ge k/2$, there is $b\in V(D')$ with 
		$\deg^-_{D'}(b)\ge k$, and for each $a\in (D')^+$ we have $\deg^+_{D}(a)> k-r$.
	\end{enumerate}
\end{lemma}

\begin{proof}
	We construct a bipartite graph $G$ as follows. For each vertex $u\in D$, we define two vertices 
	$u^+$ and $u^-$ in $G$, with $u^+$  adjacent to $v^-$ whenever $uv$ is an arc in $D$. 
	Set \[V^+:=\{v^+:v\in D\} \quad \text{and} \quad V^-:=\{v^-:v\in D\}.\] Note that $G=(V^+,V^-)$ is a bipartite graph
	of order $2n$ and size $e(G)=a(D)>(k-1)|G|/2$. 
	By Lemma~\ref{lemma:bipartite}, there is a subgraph $H$ of $G$ satisfying conditions~\ref{item:1}~and~\ref{item:2}, and 
	one of \ref{bip-i} or \ref{bip-ii}.
	
	Define subdigraph $D'$ of $D$ as follows: we let $v$ be in $D'$ whenever $v^+\in H$ or $v^-\in H$, 
	and we let $uv$ be an arc of $D'$ whenever $u^+v^-$ is an edge in $H$. 
	It is readily seen that 
	conditions \ref{cor:1} and \ref{cor:2} of Lemma~\ref{cor:subdigraph} hold, as well as one of~\ref{cor-subdigraph-i}, \ref{cor-subdigraph-ii}.
\end{proof}


\section{Embedding  $T$ when $\Delta_2\le \lfloor k/4\rfloor +2$}
\label{sec:Delta2-small}

In this section, we will prove Lemma~\ref{lemma:onepointfive}. 
This proof splits into  two lemmas, Lemma~\ref{lemma:first} and  Lemma~\ref{lemma:second}. Both lemmas are stated in Subsection~\ref{sec:onepointfive}, together with a third auxiliary result, Lemma~\ref{lemma:k/4neighbors}, which will turn out to be very useful in the remainder of the paper. The other two subsections of this sections are devoted to the proofs of Lemma~\ref{lemma:first} and  Lemma~\ref{lemma:second}.

\subsection{Proof of Lemma~\ref{lemma:onepointfive}}\label{sec:onepointfive}

Lemma~\ref{lemma:onepointfive} follows immediately from the following two lemmas.
\begin{lemma}
	\label{lemma:first}
	Let $T$ be a $k$-arc antidirected tree, and let $D$ be a \ksf digraph with more than $(k-1)|V(D)|$ arcs, where   $s=\lceil k/12 \rceil$.
	If $\Delta (T)\le \lfloor k/4\rfloor$, then $T$ embeds in $D$. 
\end{lemma}

\begin{lemma}
	\label{lemma:second}
	Let $T$ be a $k$-arc antidirected tree, and let $D$ be a \ksf digraph with more than $(k-1)|V(D)|$ arcs, where   $s=\lceil k/12 \rceil$.
	If  $\Delta (T)>\lfloor k/4\rfloor $ and $\Delta_2 (T)\le \lfloor k/4\rfloor+2$, 
	then $T$ embeds in~$D$. 
\end{lemma}

We will prove Lemma~\ref{lemma:first} in Subsection~\ref{sec:lemma-first} and Lemma~\ref{lemma:second} in Subsection~\ref{sec:lemma-second}. 

We finish this subsection by showing a quick auxiliary result, Lemma~\ref{lemma:k/4neighbors} below, which will be useful in both Subsection~\ref{sec:lemma-first} and Subsection~\ref{sec:lemma-second}.

\begin{lemma}
	\label{lemma:k/4neighbors}
	Let $D$ be a \ksf digraph, with $s=\lceil k/12\rceil$, and let $S$ be a set of at most $k$ vertices of $D$. Then, 
	for any three vertices $a,b,c\in D$ and $\star_a,\star_b,\star_c\in \{+,-\}$
	we have
	\[|N^{\star_a}(a)\cap S|+|N^{\star_b}(b)\cap S|+|N^{\star_c}(c)\cap S|<  {5k}/{4}.
	\] 
\end{lemma}
\begin{proof}
	Given that $D$ is \ksff, we have that $$|N^{\star_x}(x)\cap N^{\star_y}(y)|\le s-1=\lceil k/12 \rceil -1< k/12$$ for any $x,y\in \{a,b,c\}$. 
	Therefore, 
	\begin{align*}
		k\ge |S|&\ge |\left(N^{\star_a}(a)\cap S\right)\cup \left(N^{\star_b}(b)\cap S\right)\cup \left(N^{\star_c}(c)\cap S\right)|\\ 
		&> |N^{\star_a}(a)\cap S|+|N^{\star_b}(b)\cap S|+|N^{\star_c}(c)\cap S|-3\cdot k/12,\end{align*}
	which directly  implies the result. 
\end{proof}

\subsection{Proof of Lemma~\ref{lemma:first}}
\label{sec:lemma-first}
We start by observing that by Lemma~\ref{lempseu}, $D$ has a subdigraph $D'$ with  
\begin{equation}\label{mindeg63}\overline{\delta^0}(D')\ge \lceil k/2\rceil.\end{equation}
We will embed $T$ in $D'$. Note that since 
\begin{equation}\label{Delta63}\Delta:=\Delta(T)\le \lfloor k/4\rfloor.\end{equation}
and because of~\cref{mindeg63}, any subtree of $T$ with at most $\lceil k/2\rceil +1$ vertices can be embedded greedily in~$D'$.
Let~$T'$ be a subtree of $T$ of maximal order such that there is an embedding  $f:V(T')\rightarrow V(D')$. Then $|T'|\ge  \lceil k/2\rceil+1$. 

If $T'=T$ then we are done, so let us assume that $T'\ne T$. 
Choose $w\in T'$ and $w'\in V(T)\setminus V(T')$ such that there is an arc of $T$ containing $w$ and $w'$. We may assume that this arc is directed from $w$ to $w'$, meaning that $w$ is an out-vertex (the other case is analogous).	

From now on, we shall see $T$ as a rooted tree with $w$ as its root. 
We recall that $p_x$ denotes the parent of a vertex $x\ne w$. 
The maximality of $T'$ implies that 
$N^{+}(f(w))\subseteq f(V(T'))$. Further,  $\deg^+(f(w))\ge \lceil k/2\rceil$ by~\cref{mindeg63}, and  $\deg^+(w)\le \lfloor k/4\rfloor$ by~\cref{Delta63}, and thus, 
\[Y:=f^{-1}\left(  N^{+}(f(w))\right)\setminus N^+(w)\] contains at least $\lfloor k/4\rfloor$ vertices.	Observe that for any
$y\in Y$, we have $\dist(w,y)\ge 2$. 
Choose $y\in Y$ obeying the following two conditions:
\begin{enumerate}[label*=$(\alph*)$]
	\item $\dist(w,y)$ is maximised; and
	\item  if $\dist(w,y')=2$ for each $y'\in Y$, then   $N^{-}(f(p_{y}))\setminus f(V(T'))\neq \emptyset$. 
\end{enumerate}
Condition $(b)$  is feasible  since, if $\dist(w,y')=2$ for each $y'\in Y$, then by~\cref{Delta63} and given that $|Y|\ge \lfloor k/4\rfloor$, there are distinct $y_1,y_2\in Y$ with $p_{y_1}\ne p_{y_2}$. 
Then	 Lemma~\ref{lemma:k/4neighbors} applied  to vertices $f(p_{y_1}), f(p_{y_2})$ and $f(w)$ gives that one of $y_1,y_2$ can be chosen as $y$ in $(b)$.
Note that $y$ is an out-vertex if condition $(b)$ holds; however, we do not make any assumption on whether $y$ is an in-vertex or an out-vertex in general. 

Our goal now is to identify a subtree $T_F$ of $T$ that includes vertex  $w$ and a vertex $z$ (which we will determine in the process), and for which there exists an embedding $g_F$ of $T_F$ into $D'$.  A critical property of this subtree is that the vertices $g_F(w)$ and $g_F(z)$ together send at least $k-1$ arcs to $\im g_F$ (a set of at most $k$ vertices). By applying Lemma~\ref{lemma:k/4neighbors}, we then deduce that any other vertex $a\in V(D')\setminus \{g_F(w),g_F(z)\}$ sends fewer than $k/4+1$ arcs to $\im g_F$, and thus has at least $\lfloor k/4\rfloor -1$ neighbours outside $\im g_F$. This key observation leads to a contradiction in the choice of $T_F$. 

Let $C$ be the component of $T'\setminus \{y\}$ that contains $w$. Our choice of $y$ ensures that \begin{equation}\label{N+fw63}N^{+}(f(w))\subseteq f(C\cup \{y\}).\end{equation} 
Set 
\[C':=\begin{cases}
	C\cup \{w', y\} & \text{if 
		$\dist(w,y)=2$, }\\
	C\cup \{w'\} & \text{otherwise.}
\end{cases}\]
Let~$\mathcal{F}$ be the set of all subtrees $T_F$ of $T$ containing $C'$ such that there is an embedding $g_F:V(T_F)\rightarrow V(D')$ with 
\begin{itemize}
	\item  $g_F(w')=f(y)$, and
	\item	 $g_F(t)=f(t)$ for each $t\in C\setminus \{p_y\}$. 
\end{itemize}
Note that $T_F=C'\in \mathcal{F}$, as we can set $g_F(p_y)=f(p_y)$, and additionally, $g_F(y)\in N^{-}(f(p_{y}))\setminus \im f$ if $y\in V(C')$, 
which is feasible due to condition $(b)$.  
Therefore,  $\mathcal{F}\ne \emptyset$.

Let $T_F\in\mathcal F$ be a maximal element. 
By maximality of $T'$ and since $w'\in V(T_F)\setminus V(T')$, we know that there is a vertex~$z_F\in V(T_F)\setminus\{w\}$ whose neighbourhood in $T'$ is not completely contained in $T_F$. 
Choose $T_F$ and $z_F$ such that $\dist_T(w,z_F)$ is maximised. We will assume that $z_F$ is an out-vertex, as the other case is analogous (just change all signs at $z_F$ and $g_F(z_F)$ to $-$). 

By the maximality of $T_F$ we have
\begin{equation}
	\label{eq:neighborhood-z_i}
	N^{+}(g_F(z_F))\subseteq g_F(V(T_F)). 
\end{equation}

We claim that \begin{equation}
	\label{pzinotw}
	p_{z_F}\neq w.
\end{equation}
Indeed, by our choice of $C'$, we know that $z_F$ is either $y$ or a descendant of $y$ if $\dist(w,y)=2$. Moreover, $z_F$ is either $p_y$ or a descendant of $p_y$ if $\dist(w,y)\ge 3$. This proves~\cref{pzinotw}.

Because of~\cref{N+fw63}, we know that \[N^{+}(g_F(w))=N^{+}(f(w))\subseteq g_F(C').\]
It may occur that $z_F=p_y$ and $f(p_y)\in N^+(g_F(w))\setminus \im g_F$. Thus, using~\cref{mindeg63}, we deduce that 
\begin{equation}
	\label{eq:neighborhood-w}
	|N^{+}(g_F(w))\cap  \im g_F|\ge \lceil k/2\rceil-1. 
\end{equation}

Now, \cref{mindeg63} combined with~\cref{eq:neighborhood-z_i}, and~\cref{eq:neighborhood-w} provide us with two vertices, $g_F(w)$ and $g_F(z_F)$, such that the sum of their out-degrees  into $\im g_F$ (a set of at most $k$ vertices) is at least $k-1$. By our choice of $z_F$, these two vertices are distinct. So, we can use Lemma~\ref{lemma:k/4neighbors} to see that for every vertex $v\in V(D')\setminus \{g_F(w),g_F(z_F)\}$ 
\begin{equation}\label{allhappy63}
	|N^\star(v)\setminus \im g_F|\ge \lfloor k/4\rfloor-1,\textrm{\, where $\star\in \{+,-\}$ is such that $\deg^\star(v)\ge  \lceil k/2\rceil$.}
\end{equation}
In particular this holds for $g_F(p_{z_F})$, which by~\cref{pzinotw} is different from~$w$. So we can choose a vertex $b\in N^{-}(g_F(p_{z_F}))\setminus \im g_F$. By~\cref{allhappy63},
\begin{equation*}
	\label{eq:image-z_1}
	X:=N^{+}(b)\setminus \im g_F 
\end{equation*}
contains at least $\lfloor k/4\rfloor -1$ vertices. 
Let $T'_F$ be obtained from $T_F$ by deleting  all descendants of $z_F$. Adding to $T'_F$ the set $N_F$ of all children of $z_F$ in $T$, we obtain the tree $T''_F$. Observe that $|N_F|\le \lfloor k/4\rfloor -1$ because of \cref{Delta63}.

We change $g_F$ to an embedding $g$ of $T''_F$ in $D'$ as follows:
$g(t)=g_F(t)$ for each $t\in V(T''_F-(\{z_F\}\cup N_F))$, $g(z_F)=b$ 
and $g(N_F)\subseteq X$. 
This is feasible because of~\cref{Delta63} and our choice of $X$. Next, we extend $T''_F$ to a maximal subtree $T''$ of $T$ such that 
$g$ extends to an embedding of $T''$.   Since 
$C'\subseteq V(T'')$ (as both $y$ and $p_y$ are in $T''$), $g(w')=f(y)$ and $g(t)=f(t)$ for each $t\in C\setminus \{p_y\}$, 
it follows that $T''\in \mathcal{F}$. Moreover, by~\cref{allhappy63}, we know that $z_{T''}$ must either belong to $N_F$ or be a descendant of a vertex in $N_F$, as in $g_F$ (and thus in $g$) all other vertices that are not descendants of $z_F$ (except possibly $w$) have all their neighbours embedded. 
Therefore,  $\dist(w,z_{T''})>\dist(w,z_F)$, which contradicts the choice
of $T_F$ and $z_F$, and this in turn concludes the proof.

\subsection{Proof of Lemma~\ref{lemma:second}}
\label{sec:lemma-second}

We will need the following auxiliary lemma. 

\begin{lemma}\label{lem:case3b}
	Let $T$ be a $k$-arc antidirected tree with $\Delta_2(T)\le \lfloor k/4\rfloor +2$ whose maximum degree vertex $u$ fulfills $\deg^+ (u)>\lfloor k/4\rfloor$. Let $D$ be a \ksf digraph,  with $s=\lceil k/12\rceil$, and let $D'\subseteq D$  with $\overline{\delta^0}(D')\ge  \lceil k/2\rceil$. Let $a\in V(D')$ with $\deg_{D}^+(a)\ge \Delta (T)$.  Then~$T$ embeds in $D$, with $u$ embedded in~$a$.
\end{lemma}

Before proving Lemma~\ref{lem:case3b}, let us show how it implies Lemma~\ref{lemma:second}.

\begin{proof}[Proof of Lemma~\ref{lemma:second}]
	Set $\Delta := \Delta(T)$. Note that $D$ must contain a vertex $a$ with $\deg_{D}^+(a)\ge k\ge\Delta$, as $D$ has more than $(k-1)n$ arcs.

	We apply Lemma~\ref{cor:subdigraph} with
	$r=\min\{\lceil k/2\rceil, k-\Delta+1\}$ to find 
	a subdigraph $D'\subseteq D$ satisfying conditions~\ref{cor:1}, \ref{cor:2}, and one of \ref{cor-subdigraph-i}~or~\ref{cor-subdigraph-ii}. 
	In the latter case, we are immediately done by Lemma~\ref{lem:case3b}, since for each $a\in (D')^+$ 
	we have $\deg^+_{D}(a)\ge k-r+1\ge\Delta$. So we can assume that \ref{cor-subdigraph-i} holds.
	
	If $r=\lceil k/2\rceil$  then we are also done by Lemma~\ref{lem:case3b}, so we can assume that $r=k-\Delta+1$, that is, $k+1=r+\Delta$. Since $|V(T)\setminus N^+[u]|=r-1$, at most $r-1$ vertices from $N^+(u)$ can be non-leaves, implying the existence of at least $\Delta - r + 1$ leaves adjacent to $u$. 
	Let $T^\star$ be obtained from $T$ by removing $\Delta-r$ of these leaves. 
	Then~$T^\star$ has $k':=k-(\Delta-r)=2r-1$ arcs. Exactly $r=\lceil k'/2\rceil$ of these are adjacent to $u$. Thus
	$u$ is a vertex of maximum total degree $\Delta(T^\star)=\lceil k'/2\rceil$ in $T^\star$. 
	
	By~\ref{cor-subdigraph-i}, 
	$\overline{\delta^+}(D')\ge  \lceil k/2\rceil \ge \lceil k'/2\rceil$ and $\overline{\delta^-}(D')\ge r=\lceil k'/2\rceil$. Moreover, there is a vertex $a'\in D'$ with $\deg_{D'}^+(a')\ge k\ge k'$. 
	By Lemma~\ref{lem:case3b}
	there is an embedding of $T^\star$ in $D$, with $u$ embedded in $a'$. We complete this to an embedding of $T$ by putting the missing leaves  into unused out-neighbours of $a'$ (which is possible because $\deg^+_{D'}(a')\ge k$).  
\end{proof}

In order to prove Lemma~\ref{lem:case3b}, we state and prove another auxiliary lemma: 
\begin{lemma}\label{lem:pu}
	Let $T$ be a $k$-arc antidirected tree with $\Delta_2(T)\le \lfloor k/4\rfloor+2$ containing a vertex~$u$ of maximum degree with $\deg^+ (u)> \lfloor k/4\rfloor$, and such that every vertex of $T$ has distance at most $2$ to $u$. 
	Let $D$ be a \ksf digraph, with $s=\lceil k/12\rceil$, and let $D'\subseteq D$   with $\overline{\delta^0}(D')\ge  \lceil k/2\rceil$. Let $a\in V(D')$ with $\deg_{D}^+(a)\ge \Delta(T)$.  Then~$T$ embeds in~$D$, with $u$ embedded in $a$, and with $V(D)\setminus V(D')$ only used for images of leaves adjacent to $u$. 
\end{lemma}

\begin{proof}[Proof of Lemma~\ref{lem:pu}]
	Set $\Delta:=\Delta(T)=\deg^+(u)$ and $\Delta_2:=\Delta_2(T)$. 
	We consider $u$ to be the root of $T$. We embed $u$ into   $a$ and $N^+(u)$ into  $N^+(a)$, using vertices from $N^+(a)\setminus V(D')$ only for leaves of $T$. This is possible since $\overline{\delta^0}(D')\ge  \lceil k/2\rceil$, and at most $\lfloor k/2\rfloor$ neighbours of $u$ are not leaves. 
	
	Now consider an extension  $f$ of this embedding  that maximises the number of embedded  vertices, and does not use any further vertices from $V(D)\setminus V(D')$ (other than those already used for leaves adjacent to $u$). If $f$ fails to embed all of $V(T)$, then there is a vertex $w\in N^+(u)$ having an unembedded child $w'$. In particular, $|N^-(f(w))\cap \im f|\ge \lceil k/2\rceil$, because of our condition on the minimum pseudo-semidegree of $D'$. Since $D'$ is a \ksf digraph, less than $s$ vertices from $N^-(f(w))$ are out-neighbours of  $f(u)$, and at most $\lfloor k/4\rfloor +1$ from $N^-(f(w))$ are images of neighbours of $w$, as $\deg^+_{T}(w)\le \Delta_2\le \lfloor k/4\rfloor +2$ and $w'$ is an unembedded neighbor of $w$. So there is a vertex $y\in \dom f$ that has distance $2$ to $u$ such that $f(y)\in N^-(f(w))$.
	
	By our choice of $f$, the image of the parent, $p_y$, of $y$ has no in-neighbours outside $\im f$, as otherwise we could reembed $y$ and   embed $w'$ in $f (y)$. So $|N^-(f(p_y))\cap \im f|\ge \lceil k/2\rceil$.  Since 
	$\deg^+(u)=\Delta> \lfloor k/4\rfloor$ and all of $u$'s children were embedded, the vertices $a$, $f(w)$ and $f(p_y)$ together send more than $5k/4$ arcs to $\im f$ (a set of at most $k$ vertices). This contradicts  Lemma~\ref{lemma:k/4neighbors}.  
\end{proof}

\begin{proof}[Proof of Lemma~\ref{lem:case3b}]
	Let $\{T_i\}_{i=1}^r$ be a  family of subtrees of $T$ (rooted at $u$), with the following properties:
	\begin{itemize}
		\item $T_1$  consists of $u$ and all vertices
		of distance at most~$2$ to $u$; 
		\item $T_r=T$; and
		\item $T_{i+1}$ is obtained from $T_i$ by adding all children (in $T$) of some vertex of $T_i$.
	\end{itemize}
	We will show by induction that for each $i\in [r]$, there is an embedding of $T_i$ in $D$
	such that 
	\begin{equation}\label{indu:case3b}
		\text{$u$ is mapped to $a$, and $V(D)\setminus V(D')$ is only being used for leaves adjacent to $u$.}
	\end{equation}
	By Lemma~\ref{lem:pu}, this is true for $i=1$. 
	For the induction step, let $i\in [r-1]$. We wish to show the statement for $i+1$.

	Let $w\in V(T_i)$ be the vertex that has neighbours in $T_{i+1}-T_i$. Note that $w\notin N^+(u)\cup\{u\}$. 	Let us assume that $w$ is an out-vertex in~$T$ (otherwise just switch signs in what follows). Let $P$ be the path in $T$ joining $u$ and $p_w$ (the parent of $w$).
	
	For contradiction, assume we cannot embed $T_{i+1}$  into $D$ fulfilling~\cref{indu:case3b}. Among all choices for an embedding of $T_i$  into $D$  fulfilling~\cref{indu:case3b}, choose  $f_i$ such that 
	\begin{equation}\label{choicephii}
		\text{$f_i(w)$ has the minimum number of neighbours in $f_i(P)$,}
	\end{equation}
	and, among all choices of  $f_i$ fulfilling~\cref{choicephii}, choose $f_i$ such that 
	\begin{equation}\label{choicephi}
		\text{$f_i(w)$ has the maximum number of neighbours outside $\im f_i$,}
	\end{equation}
	Set $b_1:=f_i(w)$. Let $T'$ be a maximal subtree of $T$ with $T_{i}\subseteq T'\subseteq T_{i+1}$ 
	such that $f_i$   extends to an embedding  $f$ of $T'$ into $D$ fulfilling~\cref{indu:case3b}. 
	Since $T_1\subseteq T'$, we have that 
	\begin{equation}\label{eq:a_out}
		\text{$p_w\neq u$ \ and \ $|N^+(a)\cap \im f|\ge \deg^+(u)>\lfloor k/4\rfloor $.}
	\end{equation}
	
	Note that $T'\neq T_{i+1}$, and 
	thus  all out-neighbours of $b_1$ lie in $\im f$. As $\overline{\delta^0}(D')\ge \lceil k/2\rceil$ by assumption, we obtain that
	\begin{equation}\label{eq:b1_out}
		|N^+(b_1)\cap \im f|\ge \lceil k/2\rceil. 
	\end{equation}

	Let $B:=N^-(f(p_w))\setminus \im f$.	
	Now, if $B=\emptyset$, then $|N^-( f(p_w))\cap \im f|\ge \lceil k/2\rceil$, and so, by~\cref{eq:a_out} and~\cref{eq:b1_out}, the sum of the degrees of   vertices $a$, $b_1$ and $ f(p_w)$ into $\im f$ exceeds $5k/4$, a contradiction to
	Lemma~\ref{lemma:k/4neighbors}. Therefore,
	\begin{equation}\label{eq:B1}
		\text{$B\ne \emptyset$. }
	\end{equation}
	The same argument gives that $|B|\ge 2$, unless  $\Delta:=\deg^+(u) \le \lfloor k/4\rfloor+1$.
	Also, if $|B|=1$, then $|N^-( f(p_w))\cap \im f|\ge \lceil k/2\rceil-1$. 
	So, 
	\begin{equation}\label{eq:B2x}
		\text{if $|B|=1$ then 		$\Delta \le \lfloor k/4\rfloor +1$ and $|N^-( f(p_w))\cap \im f|\ge \lceil k/2\rceil-1$. }
	\end{equation}
	
	Next, observe that if some
	$b\in B$ has at least $\Delta_2-1$ 
	out-neighbours outside $ f(T_i\setminus\{w\})$, then we can reembed $w$ into $b$, and embed all out-neighbours of $w$ outside 
	$ f(T_i\setminus\{w\})$ (which is possible since 
	$|N^+(w)\setminus\{p_w\}|\le \Delta_2-1$). As this is an embedding of $T_{i+1}$  into $D$ fulfilling~\cref{indu:case3b}, and since $\overline{\delta^0}(D')\ge \lceil k/2\rceil$ by assumption,
	we deduce that  each $b\in B$ has more than $\lceil k/2\rceil-\Delta_2+1$ out-neighbours in $ f(T_i\setminus \{w\})$.
	So, for each $b\in B$, we have
	\begin{equation}
		\label{eq:bi_out}
		\text{
			$|N^+(b)\cap  f(T_i\setminus\{w\})| \ge  \lceil k/2\rceil-\Delta_2+2$,} 
	\end{equation} 
	which is at least $k/4$, since
	$\Delta_2\le   \lfloor k/4\rfloor  +2$ by assumption. 
	In particular, Lemma~\ref{lemma:k/4neighbors}, applied to
	$a$, $b_1$ and any $b\in B$, together with~\cref{eq:b1_out},
	implies that $|N^+(a)\cap \im  f|< \lfloor k/2\rfloor$. Thus, by our condition on the minimum degree of $D'$,
	\begin{equation}\label{eq:a_outk2}
		\text{$|N_{D'}^+(a)\setminus \im  f|> 0$.}
	\end{equation}
	
	Next, we show that
	\begin{equation}
		\label{claim:B-order}
		|B|\ge 2.
	\end{equation}
	Suppose otherwise. Then $|B|=1$ by~\cref{eq:B1}. Let $b$ be the unique vertex in $B$. By~\cref{eq:B2x}, 
	we know that $\Delta_2\le \Delta \le \lfloor k/4\rfloor+1$.
	So, ~\cref{eq:bi_out} implies that $|N^+(b)\cap \im f|\ge \lceil k/4\rceil+1$. 
	Then, by~\cref{eq:b1_out} and~\cref{eq:B2x}, vertices $b_1$, $b$ and $ f(p_w)$ together send at least 
	$5k/4$ edges to $\im f$ (which is a set of at most $k$ vertices), a contradiction to Lemma~\ref{lemma:k/4neighbors}. This proves~\cref{claim:B-order}.

	Let  $R:=V(P)\cup N_{T'}^+(w)$. 
	Since $\Delta>\lfloor k/4\rfloor$ by assumption, and since $|V(T')|\le k$, we have 
	\begin{equation}
		\label{eq:R-order}
		|R|\le k-(\Delta-1)-|\{w\}|< 3k/4.
	\end{equation}
	We claim that 
	\begin{equation}
		\label{eq:neighbors-b1-inR}
		|N^+(b_1)\cap f(R)|\ge \lceil k/2\rceil. 
	\end{equation}
	Assuming that~\cref{eq:neighbors-b1-inR} holds, we can conclude as follows.
	By our choice of $ f_i$ satisfying~\cref{choicephii}, by~\cref{eq:neighbors-b1-inR}, and  since $\deg^+_{T'}(w)\le \Delta_2-1\le \lfloor k/4\rfloor+1$ by assumption, we obtain that for each $b\in B$, 	\begin{align*}
		\label{b1-neighb}
		|N^+(b)\cap  f(V(P))|&\ge |N^+(b_1)\cap  f(V(P))|\ge \lceil k/2\rceil-\deg^+_{T'}(w)+|\{p_w\}|
		\\* &\ge \lceil k/2\rceil-\Delta_2+2\ge k/4.
	\end{align*}
	Since  $D$ is a \ksf digraph, and
	$ f(p_w)$ is an out-neighbour of each $b\in B$ and $b_1$, and because of 
	\cref{eq:neighbors-b1-inR}, we deduce that
	\begin{align*}
		|R|=| f(R)|\ge k/2+ k/4 + k/4 -
		3(s-1)+1  \ge 3k/4+1,
	\end{align*}
	a contradiction to \cref{eq:R-order}. 
	
	It remains to prove~\cref{eq:neighbors-b1-inR}. 
	By~\cref{eq:b1_out}, it suffices to show that $N^+(b_1)\subseteq  f(R)$. 
	For contradiction, suppose this does not hold, and choose a vertex $y\in V(T')\setminus R$ with $ f(y)\in N^+(b_1)$ such that $\dist(w,y)$ is maximum. 
	Let $C$ be the component of $T_i\setminus \{y\}$ that contains $P\cup\{w\}$, and let $ f_C$ be the restriction of $ f_i$ to $V(C)$. Observe that $C$ contains all neighbours of $u$, except $y$ (in the case that $y$ is a neighbour of $u$).
	
	Let $T''$ be a largest possible subtree of $T_i$ containing $C$ and $y$ such that there is an extension $ f'$  of $ f_C$ obeying~\cref{indu:case3b} with 
	$ f((N^+(w)\setminus \{p_w\})\cup\{y\})\cap \im f'=\emptyset$. If $y$ is a neighbour of $u$, such a tree $T''$ exists because of~\cref{eq:a_outk2}. 
	Otherwise, applying Lemma~\ref{lemma:k/4neighbors} to $b_1=f_i(w)$, $a=f_i(u)$ and $f_i(p_y)$ implies that $f_i(p_y)$ has an (in- or out-)neighbour outside $\im f$, in which $y$ can be embedded. Thus, such a tree $T''$ also exists in this case.  
	By our choice of $ f_i$ satisfying~\cref{choicephi}, $T''\neq T_i$, and thus there is a vertex 
	$z\in T''$ with a neighbour, say $z'$, in $T_i-T''$. 	 By our choice of $T''$, we know that $z\neq u$ and $T''$ contains all neighbours of $u$. Thus $|N_D^+(a)\cap \im  f'|> \lfloor k/4\rfloor$. 
	
	The maximality of $T''$ implies that $N^+( f'(z))$ or $N^-( f'(z))$ is contained in $\im f'$. As $\overline{\delta^0}(D')\ge \lceil k/2\rceil$, it follows that $|N^+(f'(z))\cap \im  f'|\ge \lceil k/2\rceil$ or $|N^-(f'(z))\cap \im  f'|\ge \lceil k/2\rceil$. 
	By our choice of $y$, we know that $w\neq z$, and furthermore, that $N^+(b_1)$ is contained in $Q:=\im f'\cup  f((N^+(w)\setminus \{p_w\})\cup\{y\})$. Indeed, $N^+(b_1)\subseteq f(C\cup \{y\})$, where $f(C\cup \{y\})$ is the vertex-disjoint union of $f((N^+(w)\setminus \{p_w\})\cup \{y\})$ and $f(C\cap V(T''))\subseteq \im f'$. Thus, $N^+(b_1)\subseteq Q$ as desired. Since $\overline{\delta^0}(D')\ge \lceil k/2\rceil$, we have $|N^+(b_1)\cap Q|\ge \lceil k/2\rceil$. On the other hand, note that $Q\subseteq f(V(T_{i+1})\setminus \{z'\})$, and therefore, $|Q|\le k$. Consequently, the vertices $a$, $b_1$ and $ f'(z)$ together send more than $5k/4$ edges to $Q$ (a set of at most
	$k$ vertices), which contradicts 
	Lemma~\ref{lemma:k/4neighbors}. Hence, \cref{eq:neighbors-b1-inR} holds. 
\end{proof}


\section{Embedding $T$ when $\Delta_2\ge \lfloor k/4\rfloor+3$}
\label{sec:lemma-third}

Our objective in this section is to prove Lemma~\ref{lemma:third}. That is, we will deal with the case $\Delta_2\ge \lfloor k/4\rfloor+3$. 
We begin by
identifying a specific subtree of $T$. 

\begin{definition}[Double broom $B_{uv}(T)$]\label{defdoublebroom}
	Let $T$ be an antidirected tree, and let   $u,v\in V(T)$.
	The {\it double broom $B_{uv}(T)$} is the subtree of $T$ that contains the closed neighbourhoods
	$N^{\re{u}}[u]$ and $N^{\re{v}}[v]$, along with the (unique) path of $T$ joining $u$ and $v$. 
\end{definition}

Our general strategy to show Lemma~\ref{lemma:third} consists of the following three steps: 
\begin{enumerate}[label*={\bf (\arabic*)}]
	\item We apply Lemma~\ref{cor:subdigraph} to identify a subdigraph $D'$ of $D$ that will accommodate almost all of $T$. (Our choice of $D'$ will depend on $T$.) 
	\item  We embed the double broom $B_{uv}(T)$ into $D'$, where $u,v\in V(T)$ with \linebreak
	$\deg^+(u)=\Delta(T)$ and $\deg^{\re{v}}(v)=\Delta_2(T)$. 
	\item  We embed $T-B_{uv}(T)$. 
\end{enumerate}
Step 1 is done in Subsection~\ref{subs:step1},  Step 2 is executed in Subsection~\ref{subs:double-broom}, while
step 3 is accomplished in Subsection~\ref{subs:extension}. 

Before continuing, let us establish some notation that will be used throughout the remainder of the paper. 

\begin{definition}[$\LL$ and $\NL$]
	We denote by $\LL$ and $\NL$ the set of leaves and non-leaves of $T$. 
	Given a subset $X\subseteq  V(T)$, we let $\LL_X$ and $\NL_X$ denote the subsets of leaves and non-leaves of $T$ in $\cup_{x\in X}N^{\re{x}}(x)$, 
	respectively. If $X=\{x\}$ we simply write $\LL_x$ and $\NL_x$. 
\end{definition}

\subsection{Finding $D'$}\label{subs:step1}\label{rem:cases}
We perform the first step of the strategy outlined above. 
Given a digraph $D$ with $n$ vertices and more than $(k-1)n$ arcs, and a $k$-arc antitree $T$ with vertices $u$, $v$ such that $\Delta:=\Delta(T)=\deg^+(u)$, $\Delta_2:=\Delta_2(T)=\deg^{\re{v}}(v)\ge \lfloor k/4\rfloor+3$ and $B_{uv}:=B_{uv}(T)$,
we define a subdigraph $D'$ as follows: 
\begin{enumerate}[label={\bf(\Alph*)}]
	\item \label{rem:1} If at least one of the following holds:
	\begin{itemize}
		\item
		$|B_{uv}|\le{3k}/{4}$, or 
		\item  $v$ is an in-vertex, $\Delta+\Delta_2<{7k}/{12}$,   and $|B_{uv}|\le k+1-(\Delta-\Delta_2)$, 
	\end{itemize}
	then we apply Lemma~\ref{cor:subdigraph} with $r=\lceil k/2\rceil$ 
	to obtain 
	a subdigraph $D'$ of $D$ with $\overline{\delta^0}(D')\ge \lceil k/2\rceil$ satisfying conditions~\ref{cor:1}~and~\ref{cor:2} of the lemma. 
	\item \label{rem:2} Otherwise, we apply Lemma~\ref{cor:subdigraph} with 
	$r=\min\{\lceil 5k/12\rceil,k-\Delta\}$ 
	to obtain a subdigraph $D'$ of~$D$
	satisfying~\ref{cor:1}~and~\ref{cor:2}, and such that 
	\begin{enumerate}[label={\bf(B-\Roman*)}]
		\item \label{rem:2.1} condition~\ref{cor-subdigraph-i} holds, or
		\item \label{rem:2.2} condition~\ref{cor-subdigraph-ii} holds.
	\end{enumerate}
\end{enumerate}

\begin{definition}[Suitable embedding]
	In case \ref{rem:2.2}, we say that an embedding $f:V(T)\rightarrow V(D)$
	is \textit{suitable} if whenever $x\in T$ is a non-leaf, we have $f(x)\in V(D')$. 
\end{definition}

\subsection{Embedding the double broom $B_{uv}$} 
\label{subs:double-broom}

In this subsection we perform step 2 of our strategy: we show that there is an embedding 
of~$B_{uv}$ into~$D'$ in cases~\ref{rem:1} and \ref{rem:2.1}, and a suitable embedding of $B_{uv}$ into $D$ in case~\ref{rem:2.2}. This is achieved in the following result.

\begin{lemma}\label{lemma:thirdpart2}
	Let $T$ be  a $k$-arc antidirected tree with $\Delta_2(T)\ge \lfloor k/4\rfloor +3$, and let $D$ be a \ksf digraph with $n$ vertices and more than $(k-1)n$ arcs, where $s=\lceil k/12\rceil$. 
	Let $u,v\in V(T)$ with 
	$\deg^+(u)=\Delta(T)$ and $\deg^{\re{v}}(v)=\Delta_2(T)$ and let $D'$ be the subdigraph from Subsection~\ref{rem:cases}. 
	Then there is an embedding of $B_{uv}(T)$ in $D'$ in Cases~\ref{rem:1}~and~\ref{rem:2.1}, and there is a suitable embedding
	of $B_{uv}(T)$ in $D$ in Case~\ref{rem:2.2}. 
\end{lemma}
\begin{proof}
	For brevity, we write $\Delta, \Delta_2$ instead of $\Delta(T), \Delta_2(T)$, and $B_{uv}$ instead of $B_{uv}(T)$.
	We proceed according to the cases  of Subsection~\ref{rem:cases}. 
	\begin{itemize}
		\item {\bf Case \ref{rem:1}:}  (I) $|B_{uv}|\le 3k/4$, or  (II) $v$ is an in-vertex, $\Delta+\Delta_2< 7k/12$,   and $|B_{uv}|\le k+1-(\Delta-\Delta_2)$.

		First assume subcase (II). As $\Delta\ge \Delta_2\ge \lfloor k/4\rfloor +3$ and $\Delta + \Delta_2 < 7k/12$, we have $d:=\Delta-\Delta_2<k/12$ and $|T\setminus B_{uv}|\ge d$. 
		We consider a double broom $B'_{uv}$ obtained 
		from $B_{uv}$ by adding $d$ in-neighbours to $v$. Observe that $|B'_{uv}|\le k+1$ and that $|(B'_{uv})^+|=|(B'_{uv})^-|$.  
		So, regardless whether $|(D')^+|\le |(D')^-|$ or $|(D')^+|> |(D')^-|$, Proposition~\ref{prop:catmindeg} (which we can apply since  $D'$ satisfies condition~\ref{cor:1} of Lemma~\ref{cor:subdigraph}) provides an embedding 
		of $B'_{uv}$ and thus of $B_{uv}$ in $D'$.

		Suppose now we are in subcase (I), i.e.~$|B_{uv}|\le 3k/4$. 			Note that  
		$\Delta<\lfloor k/2\rfloor$, as by assumption, $\Delta_2\ge \lfloor k/4\rfloor+3$.  
		Recall that since we are in Case  \ref{rem:1},  we chose $D'$ in  Subsection~\ref{rem:cases} such that $\overline{\delta^0}(D')\ge \lceil k/2\rceil$. 		
		Embed $u$ in an arbitrary vertex $a\in D'$, and embed $N^+(u)$ into $N^+(a)$. 
		Let $T'\subseteq B_{uv}$ be a maximal tree with $N^+[u]\subseteq T'$ such that there is an embedding $f:V(T')\rightarrow V(D')$. 
		If $T'=B_{uv}$ we are done, so let us assume that $T'\ne B_{uv}$. Then there is a vertex $z\in T'$ having a neighbour $z'\in V(B_{uv})\setminus V(T')$. 
		Note that	
		$z\ne u$ and $N^{\re{z}}(f(z))\subseteq f(V(T'))$. 
		Observe that if $|N^{\re{z}}(f(z))\cap f(N^+(u))|\le 1$, then 
		\begin{align*}
		|B_{uv}|>|T'|&\ge |N^{\re{z}}(f(z))|+|N^+(u)|-1\\&\ge \lceil k/2\rceil + \Delta -1\ge \lceil k/2\rceil + \Delta_2-1
		\\& >{3k}/{4},
		\end{align*}
		contradicting the hypothesis $|B_{uv}|\le 3k/4$.  		
		Thus,  $N^{\re{z}}(f(z))\cap f(N^+(u))$ contains a vertex which is the image of a leaf $h$ of $B_{uv}$ adjacent to $u$. 
		The maximality of $T'$ implies 
		that there is no alternative for embedding $h$, and therefore,  $N^+(f(u))\subseteq f(V(T'))$. Hence, 
		\[|B_{uv}|>|T'|\ge |N^{\re{z}}(f(z))|+| N^+(f(u))|-(s-1)>3k/4,\]
		a contradiction as $|B_{uv}|\le 3k/4$ by hypothesis.

		\item {\bf Case \ref{rem:2.1}:} $|B_{uv}|>3k/4$ and  $D'$ satisfies conditions~\ref{cor:1},~\ref{cor:2}~and~\ref{cor-subdigraph-i} of Lemma~\ref{cor:subdigraph}. 
		
		Condition~\ref{cor-subdigraph-i}  ensures that $|(D')^+|\le |(D')^-|$. Moreover, since $\Delta\ge\Delta_2$ and since $u$ is an out-vertex
		we know that $|B_{uv}^+|\le |B_{uv}^-|$. So Proposition~\ref{prop:catmindeg} guarantees
		that $B_{uv}$ embeds in $D'$.

		\item{\bf Case \ref{rem:2.2}:} $|B_{uv}|> 3k/4$, $D'$ satisfies conditions~\ref{cor:1},~\ref{cor:2}~and~\ref{cor-subdigraph-ii} of Lemma~\ref{cor:subdigraph}, and  for each $a\in (D')^+$ we have $\deg^+_{D}(a)> k-r$.
		
		Note that since $D'$ satisfies condition~\ref{cor-subdigraph-ii}, there is a vertex $b\in D'$ with 
		\begin{equation}\label{bwithindegk}
			\deg^-_{D'}(b)\ge k.
		\end{equation}
		We will also frequently use that 
		\begin{equation}\label{semideg_again}
			\overline{\delta^0}(D')\ge \lceil k/2\rceil.
		\end{equation}
		
		Consider first the case when $r=k-\Delta< \lceil 5k/12\rceil$, which implies that $\Delta\ge  7k/12$, 
		and further, that every vertex in $(D')^+$ has out-degree greater than $\Delta$ in $D$. 
		We set $f(u)=a$ for an arbitrary $a\in (D')^+$, and map $N^+(u)$ into $N^+(a)$ so that $f(\NL_u)\subseteq D'$. 
		Since $D$ is \ksff, for each $b\in D'$, we have $|N^\star(b)\cap f(N^+[u])|\le  \lceil k/12\rceil$ for all $\star\in \{+,-\}$. 
		Furthermore, by~\cref{semideg_again}, for each vertex $b\in D'$ we have,  
		\[|N_{D'}^\star(b)\setminus f(N^+[u])|\ge  \lceil k/2\rceil-\lceil k/12\rceil\ge \lfloor 5k/12\rfloor\ge r,\] where $\star\in \{+,-\}$ is such that $\deg_{D'}^\star(b)\ge \lceil k/2\rceil$. Since  $|B_{uv}\setminus N^+[u]|\le k-\Delta=r$, we can greedily extend the embedding of $N^+[u]$ to 
		cover all of $B_{uv}$, using only vertices of~$D'$ in the extension.   
		
		So we can assume  that $r=\lceil 5k/12\rceil$, so $\Delta\le \lfloor 7k/12\rfloor$ and  
		\begin{equation}\label{degaD'712}
			\text{$\deg_{D}^+(a)\ge k-r+1\ge  \lceil 7k/12\rceil \ge \Delta$ for each $a\in (D')^+$.} 
		\end{equation}
		
		This makes it easy to embed $B_{uv}$ if it is a double-star. Indeed, we let $a$ be any in-neighbour of $b$, where $b$ is the vertex from~\cref{bwithindegk}. Set $f(u)=a$ and $f(v)=b$. 
		Embed $N^+(u)\setminus \{v\}$ into $N_{D}^+(a)\setminus \{b\}$ such that $f(\NL_u)\subseteq D'$, which is possible by~\cref{semideg_again}~and~\cref{degaD'712}. 
		Finally, embed $N^-(v)\setminus \{u\}$ into unused vertices of $N^-_{D'}(b)$, which is possible by~\cref{bwithindegk}. It is straightforward to 
		see that $f$ is a suitable embedding of $B_{uv}$ in $D$. Hence, we can assume that 
		\begin{equation}\label{notdoublestar}
			\text{$B_{uv}$ is not a double-star. }
		\end{equation}
		
		Based on the key properties of the vertices $u$ and $v$, we unify our approach by organising the argument into several cases that reflect these properties. To streamline the analysis, we relabel $\{u,v\}$ as $\{x,y\}$ so that we can identify certain specific properties. At this stage, we are not concerned with which of $u$ or $v$ is the vertex of highest degree and which is the second highest, and instead, just use the fact that both  have relatively high degree.

		Define $x,y\in\{u,v\}$  as follows:
		\begin{enumerate}
			\item[$(i)$] if $\Delta-\Delta_2\ge k/12$, we let $x=u$ and $y=v$; 
			\item[$(ii)$] if $\Delta-\Delta_2< k/12$ and $v$ is an out-vertex, we choose $y\in\{u,v\}$ such that $|\LL_y|\ge k/12$, 
			and let $x$ be the other one of $u,v$;
			\item[$(iii)$] if $\Delta-\Delta_2<k/12$ and $v$ is an in-vertex, we let $x=v$ and $y=u$. 
		\end{enumerate}
		Note that the choice of $y$ in case $(ii)$ is possible, as $|B_{uv}|\ge 3k/4$ and $\Delta_2\ge \lfloor k/4\rfloor +3$. Also note that  in
		case~$(i)$,  $x=u$  and $\Delta-\Delta_2\ge k/12$, and so, since $|B_{uv}|\ge 3k/4$, 
		\begin{equation}\label{caseiLx}
			\text{in case~$(i)$, we have that $|\LL_x|>k/12$.}
		\end{equation}
		We start by embedding $T_1:=N^{\re{x}}[x]$. 
		In cases $(i)$ and $(ii)$, we embed $T_1$ as follows: first, we embed $x$ in a vertex $a\in D'$
		such that $\deg^+_{D}(a)\ge \lceil 7k/12\rceil$, and then we embed $N^+(x)$ in $N^+(a)$, using $D-D'$ only for  vertices from $\LL_x$. 
		This is possible by~\cref{semideg_again} and~\cref{degaD'712}.
		In case $(iii)$, we embed   $x$ in  the vertex $b$ from~\cref{bwithindegk}, and embed $N^-(x)$ in $N^-(a)$. 		
		In either case, we found a suitable embedding of $T_1$, and in case $(iii)$, we have ensured that 
		\begin{equation}
			\label{eq:degree-k}
			\parbox{0.85\linewidth}{\textit{vertex $v=x$ is embedded in a vertex in $D'$ whose in-degree (in $D'$) is at least $k$,}} 
		\end{equation}
		while in cases $(i)$ and $(ii)$, we know that 
		\begin{equation}
				\label{eq:eeeee} 
				\deg^+_{D}(f(x))\ge \lceil 7k/12\rceil.
		\end{equation}

		Let $T'$ be a maximal subtree of $B_{uv}$ containing $T_1$ such that
		there is a  suitable   embedding $f:V(T')\rightarrow V(D)$, and such that if we are in case $(iii)$ then condition~(\ref{eq:degree-k}) holds, and otherwise condition~(\ref{eq:eeeee}) holds. We root $T'$ at $x$.
		Let $N_x$ be denote the subset of vertices in $N^{\re{x}}(x)$ which are not in the path joining $x$ and $y$. Since $\deg^{\re{x}}(x)\ge \Delta_2\ge \lfloor k/4\rfloor +3$, we have 
		\begin{equation}\label{eq:violet}
			|N_x|=|N^{\re{x}}(x)|-1\ge \lfloor k/4\rfloor +2.
		\end{equation}
		
		Our next objective is to establish a lower bound on $|T'|$. This lower bound shows that all of $B_{uv}$ has been embedded except for some leaves (in $B_{uv}$) that are adjacent to $y$. In other words,$B_{uv}\setminus T'$ consists of some children of $y$.

		\begin{claim}
			\label{claim:maximum-i}
			$|T'|>3k/4$. 
		\end{claim}

		\begin{myproof}
			For the sake of contradiction suppose that $|T'|\le {3k}/{4}$. 
			Then there are $z\in T'$ and $z'\in N_T^{\re{z}}(z)\setminus V(T')$.  The maximality of $T'$ implies that
			\begin{equation}\label{eqqqqx}
				N_{D'}^{\re{z}}(f(z))\subseteq f(T').\end{equation}
			and
			\begin{equation}\label{eqqqq}
				\text{if $N_{D'}^{\re{z}}(f(z))\cap f(N_x)\neq\emptyset$ then $N_{D'}^{\re{x}}(f(x))\subseteq f(T')$.} 
			\end{equation}
			Indeed, \cref{eqqqq} holds  since otherwise we could reembed a vertex $w$ from $N_x$ into a suitable vertex outside $f(T')$, and use $f(w)$ for $z'$, a contradiction to the maximality of $T'$. 
			
			Next, note that if $N_{D'}^{\re{z}}(f(z))\cap f(N_x)$ is empty, then by~\cref{semideg_again} and by~\cref{eq:violet}, we have
			\[|f(V(T'))|\ge \deg_{D'}^{\re{z}}(f(z))+|N_x| \ge  \lceil k/2\rceil+ \lfloor k/4\rfloor +2>{3k}/{4},\]
			contrary to our hypothesis that $|T'|\le 3k/4$. 
			So $N_{D'}^{\re{z}}(f(z))\cap f(N_x)\neq\emptyset$ and thus
			by~\cref{eqqqq},
			\begin{equation}\label{eqqqqxx}
				N_{D'}^{\re{x}}(f(x))\subseteq f(T').
			\end{equation}

			By~\cref{semideg_again},~\cref{eqqqq} and~\cref{eqqqqxx},
			\[|f(V(T'))|\ge \deg_{D'}^{\re{x}}(f(x))+\deg_{D'}^{\re{z}}(f(z))-(s-1)>{11k}/{12},\]
			resulting in a contradiction. This concludes the proof of the claim. 
		\end{myproof}

		By Claim~\ref{claim:maximum-i}, and as by assumption, $\Delta_2\ge \lfloor{k}/{4}\rfloor+3$, we know that $y\in V(T')$.
		If $T'=B_{uv}$ we are done, so we can assume there is a vertex 
		$y'\in {N^{\re{y}}(y)\setminus V(T')}$. Take $y'\in \LL_y$ if possible. 
		Setting 
		\[N^{\re{y}}(f(y)):=\begin{cases}
			N_{D}^{\re{y}}(f(y)) & \text{if $y'\in \LL_y$,} \\
			N_{D'}^{\re{y}}(f(y)) & \text{if $y'\in \NL_y$,} \\
		\end{cases} \]
		we have
		\begin{equation}
			\label{eq:neighbor-y}	
			N^{\re{y}}(f(y))\subseteq f(V(T')). 
		\end{equation}
		Note that whether $y'\in \LL_y$ or $y'\in \NL_y$ influences the lower bound on the number of neighbours of $f(y)$  included in $\im f$. Specifically, if $y'\in \LL_y$, we have $|N^{\re{y}}(f(y))\cap \im f|\ge \lceil 7k/12\rceil$  by~\eqref{degaD'712}, and if $y'\in \NL_y$, we obtain $|N^{\re{y}}(f(y))\cap \im f|\ge \lceil k/2\rceil$   by~\eqref{semideg_again}. This is precisely why we chose $y' \in \LL_y$ if possible, as this choice maximises the number of neighbors of $f(y)$ within $\im f$. 
		
		Next, we show that $N^{\re{y}}(f(y))$ and $f(N_x)$ are two disjoint subsets, both contained in $\operatorname{Im} f$, by~\eqref{eq:neighbor-y} and because of $N^{\re{x}}(x)\subseteq V(T')$. A key property of these subsets is their size: $|N^{\re{y}}(f(y))|\ge \lceil k/2\rceil$ by~\eqref{semideg_again}, and $|f(N_x)|\ge \lfloor k/4\rfloor + 2$ by~\eqref{eq:violet}. These subsets will play an important role in the subsequent stages of the proof.

		\begin{claim}
			\label{claim:intersection}
			$N^{\re{y}}(f(y))\cap f(N_x)=\emptyset$. 
		\end{claim}
		\begin{myproof}
			If case $(iii)$ holds, then by~\cref{eq:degree-k}, $f(x)$ has $k$ neighbours to choose from, and therefore we may assume that $N^{\re{y}}(f(y))\cap f(N_x)=\emptyset$, 
			which is precisely the claim. Indeed, if there exists a vertex $x'\in N_x$ such that $f(x')\in N^{\re{y}}(f(y))$, we could embed $y'$ into $f(x')$, and then reembed $x'$ into an unused vertex from $N_{D'}^{\re{x}}(f(x))$. This is feasible as $\deg_{D'}^{\re{x}}(f(x))\ge k$, thereby providing an embedding of $T'\cup \{y'\}$, and contradicting the maximality of $T'$. Thus, we can assume that either case $(i)$ or case $(ii)$ holds. Observe that 
			in either instance, $x$ is an out-vertex. 
			Suppose the claim is false, and let $x'\in N_x$ with $a:=f(x')\in N^{\re{y}}(f(y))$. We take $x'\in \LL_x$ if possible. 
			
			Note that if $|N_{D}^+(f(x))\cap f(V(T'))|\ge \lceil 7k/12\rceil$, by~\eqref{semideg_again} and since $D$ is \ksff, we have
			\[|f(V(T'))|\ge |N_{D}^+(f(x))\cup N_{D'}^{\re{y}} (f(y))| \ge \lceil 7k/12\rceil +\lceil k/2\rceil-(s-1)>k,\]
			a contradiction as $T'\neq  T$, and so $|f(V(T'))|\le k$. Therefore, we must have
			\begin{equation}\label{eq:extra}
				|N_{D}^+(f(x))\cap f(V(T'))|< \lceil 7k/12\rceil.
			\end{equation}
			By~\eqref{eq:eeeee}, there exists a vertex  $b\in N_{D}^+(f(x))\setminus f(V(T'))$.  
			If $b\in D'$, or if  $x'\in \LL_x$, then we can reembed $x'$ into $b$, and embed $y'$ into $a$. This provides an embedding of $T'\cup \{y'\}$, which is a contradiction to the maximality of $T'$. 
			So we can assume that 				
			\begin{equation}
				\label{eq:x-neighborhood}					
				N^+_{D'}(f(x))\subseteq f(V(T')), 
			\end{equation}
			and that $x'\notin \LL_x$. 
			
			Suppose case~$(i)$ holds. Observe that if all vertices in $\LL_x$ are embedded in vertices of $D-D'$, by~\eqref{caseiLx}~and~\eqref{eq:x-neighborhood}, we would have $|N_{D}^+(f(x))\cap f(V(T'))|\ge \lceil 7k/12\rceil$, a contradiction to~\eqref{eq:extra}. 
			Therefore, there must exist $x^*\in \LL_x$ such that $f(x^*)\in N^+_{D'}(f(x))$. We can then redefine the embedding $f$ so that $x^*$ is embedded into $a=f(x')$ and $x'$ into $f(x^*)$. In this updated embedding, we can select $x^*\in \LL_x$ with $a=f(x^*)\in N^{\re{y}}(f(y))$, allowing us to proceed as above. This completes the proof for case~$(i)$. 

			Suppose now that case~$(ii)$ holds. 
			Observe that if $y'\in \LL_y$, then $N^{\re{y}}(f(y))=N^{\re{y}}_{D}(f(y))$, 
			and by~\cref{degaD'712}~and~\cref{eq:neighbor-y}, 
			we have $|N^{\re{y}}(f(y))\cap f(V(T'))|\ge \lceil 7k/12\rceil$. So by~\cref{eq:x-neighborhood}, we obtain 
			\[|f(V(T'))|\ge |N^+_{D'}(f(x)) \cup N_{D}^{\re{y}}(f(y)) |\ge \lceil k/2\rceil+ \lceil 7k/12\rceil -(s-1)>k,\]
			a contradiction as $T'\ne T$, and so $|f(V(T'))|\le k$. Therefore, we must have $y'\in \NL_y$, and furthermore, there exists a vertex $b\in N^{\re{y}}(f(y))\setminus f(V(T'))$. By the maximality of $T'$, we must have  $b\notin V(D')$.  
			Since $y'\in \NL_y$, all vertices in $\LL_y$ have already been embedded, and given that
			$|\LL_y|\ge k/12$ by the choice of $y$, there must exist a vertex $y^*\in \LL_y$ such that $f(y^*)\in V(D')$. Hence, we can reembed 
			$y^*$ into $b$ and embed $y'$ into $f(y^*)$. This corresponds to an embedding of $T'\cup\{y'\}$, which contradicts the maximality of $T'$.  
		\end{myproof}
		
		Let $p_y$ be the parent of $y$. By~\cref{notdoublestar}, $p_y\ne x$, and thus 
		\begin{equation*}
			N^{\re{p_y}}_{D'}(f(p_y))\setminus f(V(T'))\ne \emptyset,
		\end{equation*}
		as otherwise, using~\cref{semideg_again} for $f(y)$ and $f(p_y)$,   we obtain
		\begin{align*}
			|f(V(T'))|&\ge \deg_{D'}^{\re{y}}(f(y))+\deg_{D'}^{\re{p_y}}(f(p_y))+\deg_T^{\re{x}}(x)-3(s-1)\\&\ge k+\Delta_2-3(s-1)>k,
		\end{align*}
		a contradiction as  $T'\neq T$, and so $|f(V(T'))|\le k$. 
		Thus, there is a vertex  \[b\in N^{\re{p_y}}_{D'}(f(p_y))\setminus f(V(T')). \]
		
		Consider an alternative maximal subtree $T''$ of $B_{uv}$ such that 
		there exists a suitable embedding $g:V(T'')\to V(D)$ satisfying the following conditions: 
		$g(t)=f(t)$ for each $t\in T'\setminus \{y\}$ that is not a leaf in $B_{uv}$, 
		and $g(y)=b$. Because of the maximality of $T'$, 	we may assume  that $T''\subseteq T'$, and therefore $y'\notin T''$. Let 
		\[N^{\re{y}}(b):=\begin{cases}
			N_{D}^{\re{y}}(b) & \text{if $y'\in \LL_y$,} \\
			N_{D'}^{\re{y}}(b) & \text{if $y'\in \NL_y$.} \\
		\end{cases} \]
		Let $H:=f(V(T')\setminus (N^{\re{y}}[y]\setminus \{p_y\}))$, 
		$R_1:=N^{\re{y}}(b)\setminus H$, $R_2:=N^{\re{y}}(b)\cap f(N_x)$, and let $r_1:=|R_1|$ and $r_2:=|R_2|$. Since $R_1\cap V(H)=\emptyset$ and $R_2\subseteq f(N_x)\subseteq V(H)$, it follows that $R_1$ and $R_2$ are disjoint. 
		
		Observe that \begin{equation}
			\label{eq:r_1<}
			r_1< \deg^{\re{y}}(y)-1,
		\end{equation}
		as otherwise we could reembed $y$ into $b$, and (re)embed the neighbours of $y$, except $p_y$, into $R_1$, obtaining an embedding of all $B_{uv}$,  which
		contradicts the maximality of $T'\ne B_{uv}$.   In particular, inequality~\eqref{eq:r_1<} provides an immediate upper bound on $|N^{\re{y}}(b)\cap \im f|$, the number of arcs that $b$ sends to $\im f$. Additionally, the set $N^{\re{y}}(b)\cap \im f$, along with $N^{\re{y}}(f(y))$ and $f(N_x)$, will play an important role in the remainder of the proof. With this in mind, we next provide an upper bound on $r_1+r_2$.

		\begin{claim}
			\label{claim:second-optionx}
			We have $r_1+r_2< \deg^{\re{y}}(y)-1$. 
		\end{claim}
		\begin{myproof}
			For the sake of a contradiction, suppose that 
			\begin{equation}
				\label{eq:r_1+r_2}
				r_1+r_2\ge \deg^{\re{y}}(y)-1.
			\end{equation} 			
			First, we consider case $(iii)$. Let $P_{uv}$ be the (unique) path joining $u$ and $v$ in $T$. 
			Define an embedding $g$ as follows: $g(t)=f(t)$ for each $t\in P_{uv}\setminus \{y\}$, 
			$g(y)=b$,  $N^{\re{y}}(y)$ is embedded in 
			vertices of $R_1\cup R_2$ (which is feasible because of~\cref{eq:r_1+r_2}), and $N^-(x)\setminus V(P_{uv})$ is embedded in unoccupied 
			vertices of $N^-(g(x))$ (recall that $g(x)=f(x)$ and so by~\cref{eq:degree-k}, $\deg^-(g(x))\ge k$). 
			It is straightforward to see that $g$ is a suitable embedding of $B_{uv}$, and furthermore, 
			condition~\cref{eq:degree-k} is fulfilled; this contradicts the maximality of $T'\ne B_{uv}$.    We can thus assume we are in  case $(i)$ or in case $(ii)$. 
			
			Suppose now that $|\LL_x|\ge {k}/{12}$, which by~\cref{caseiLx} is true in case $(i)$, but may also occur in case~$(ii)$. 
			Choose a subset $R^*_2\subseteq R_2$ such that $|R_1|+|R_2^*|=\deg^{\re{y}}(y)-1$, which 
			is possible by~\cref{eq:r_1+r_2}. Observe that $R^*_2\ne \emptyset$ by \cref{eq:r_1<}. 
			We define an embedding $g$ as follows: $g(t)=f(t)$ for each $t\in f^{-1}(H\setminus R^*_2)$, $g(y)=b$, 
			$N^{\re{y}}(y)$ is embedded in vertices of $R_1\cup R^*_2$, and a maximum set~$N'_x$ of vertices from 
			$N_x$ is embedded in unused vertices of $N_{D}^{\re{x}}(g(x))$, 
			with the vertices of~$\NL_x$  embedded in vertices of $N_{D'}^{\re{x}}(g(x))$ (possible as $|\LL_x|\ge {k}/{12}$). 
			Let $T^*$ be the corresponding  subtree of $T$. The maximality of $N'_x$ implies that $N^{\re{x}}_{D}(g(x))\subseteq g(V(T^*))$.  
			Moreover, by the definition of $H$, $g$ and $T^*$, note that $N^{\re{y}}_{D'}(b)\subseteq g(V(T^*))$, 
			because $N^{\re{y}}_{D'}(b)\cap H \subseteq g(V(T^*))$ and $N^{\re{y}}_{D'}(b)\setminus H=R_1\subseteq g(V(T^*))$.
			Thus, by~\cref{eq:eeeee}, we  have
			\[|g(V(T^*))|\ge |N^{\re{x}}_{D}(g(x))\cup N^{\re{y}}(b)|\ge \lceil 7k/12\rceil +\lceil k/2\rceil-(s-1)>k,\]  
			a contradiction since $T^*\neq B_{uv}$ by the maximality of $T'$, and so $|T^*|\le k$. 
			
			Finally, suppose that  $|\LL_x|<{k}/{12}$. Then $(ii)$ holds and,  because of our bound on $\Delta_2$, we have $|\NL_x|>{k}/{6}$. 
			Thus 
			\begin{equation}\label{Buv5k6}
				|B_{uv}|\le {5k}/{6}. 
			\end{equation}
			Since we are in case~$(ii)$,  $|\LL_y|\ge {k}/{12}$, 
			and so we used up all of $N^{\re{y}}_{D-D'}(f(y))$ when embedding~$\LL_y$. Hence  $N^{\re{y}}_{D}(f(y))\subseteq f(V(T'))$. 
			Moreover, by definition of $H$ and $R_1$ we have 
			$N^{\re{y}}_{D}(b)\setminus f(V(T'))\subseteq R_1$. So,  by~\cref{degaD'712} and~\cref{eq:r_1<},
			\[|N^{\re{y}}_{D}(b)\cap f(V(T'))|\ge \lceil 7k/12\rceil -r_1> \lceil 7k/12\rceil -\deg^{\re{y}}(y)+1.\] 
			Combining this fact
			with Claim~\ref{claim:intersection}, and using~\cref{degaD'712} and~\cref{eq:eeeee}, we obtain the following inequality: 
			\begin{align*}
				|f(V(T'))| & \ge |N^{\re{y}}_{D}(f(y))\cup N^{\re{y}}(b)\cup N^{\re{x}}_{D}(f(x))| \\ & \ge \lceil 7k/12\rceil+\lceil 7k/12\rceil-\deg^{\re{y}}(y)+1\\& \qquad +\deg^{\re{x}}(x)-1-2(s-1).\end{align*}
			Regardless of how $x$ and $y$ were selected (from the set $\{u,v\}$), we have  
			\[\deg^{\re{x}}(x)-\deg^{\re{y}}(y)> -{k}/{12}\] 
			because we are in case $(ii)$ and therefore $\Delta-\Delta_2<{k}/{12}$. 
			Thus, 
			\[|f(V(T'))|> {13k}/{12}-2(s-1)> {11k}/{12};\]
			contrary to~\cref{Buv5k6} as $T'\subseteq B_{uv}$. 
			This completes the proof of the claim. 
		\end{myproof}

		By~\cref{semideg_again} and by~\cref{degaD'712}, and using 
		Claim~\ref{claim:second-optionx} as well as the definition of $r_1$, we obtain that		
		\begin{equation*}
			|N^{\re{y}}(b)\cap \im f|\ge \begin{cases}
				\lceil 7k/12\rceil-\deg^{\re{y}}(y)+r_2+1 & {\parbox[t]{.25\textwidth}{if $y$ is an out-vertex and $|\LL_y|\ge {k}/{12}$,}}\\
				\lceil k/2\rceil-\deg^{\re{y}}(y)+r_2+1 & \textrm{otherwise.} \\
			\end{cases}
		\end{equation*}

		By Claim~\ref{claim:intersection}, the following holds. 
				\begin{align*}
		|f(V(T'))|&\ge |N^{\re{y}}(f(y))|+|N^{\re{y}}(b)\cap f(V(T'))|+|f(N_x)|\\
		&\qquad -|N^{\re{y}}(f(y)) \cap N^{\re{y}}(b)|-r_2. 
		\end{align*}
		We will use these two facts to finish the proof, distinguishing several cases, always using~\cref{semideg_again} and \cref{degaD'712}. 
		\begin{itemize}
			\item Suppose $(i)$ holds, so we have $\Delta-\Delta_2\ge{k}/{12}$, $x=u$ and $y=v$. Then, 
			\begin{align*}
			|f(V(T'))|&\ge \lceil k/2\rceil+\left(\lceil k/2\rceil-\Delta_2+r_2+1\right)+(\Delta -1)-(s-1)-r_2\\ &\ge k+1, 
		\end{align*}
			a contradiction. 
			\item Suppose $(ii)$ holds, so 
			$\Delta-\Delta_2<{k}/{12}$  (and therefore $\deg(x)-\deg(y)\ge -k/12$), and further, 
			$y$ is an out-vertex with $|\LL_y|\ge {k}/{12}$. 
			So, 
			\begin{align*}
			|f(V(T'))|&\ge \lceil 7k/12\rceil+\left(\lceil 7k/12\rceil-\deg^{\re{y}}(y)+r_2+1\right)\\ &\qquad +\left(\deg^{\re{x}}(x)-1\right) -(s-1)-r_2 \\ & \ge k+1, 
		\end{align*}
			a contradiction. 
			\item Suppose $(iii)$ holds and that $|\LL_y|\ge {k}/{12}$, so 
			$\Delta-\Delta_2<k/12$ and $y=u$ is an out-vertex. As in the previous case, we have 
			\begin{align*}
			|f(V(T'))|& \ge \lceil 7k/12\rceil+\left(\lceil 7k/12\rceil-\deg^{\re{y}}(y)+r_2+1\right)\\ &\qquad+\left(\deg^{\re{x}}(x)-1\right)-(s-1)-r_2\\ &\ge k+1, 
		\end{align*}
			a contradiction. 
			\item Suppose $(iii)$ holds and that $|\LL_y|<k/12$. Then
			$\Delta-\Delta_2<k/12$, 
			$y=u$ and, because of our bound on $\Delta_2$, we have $|\NL_y|>k/6$,
			and thus $|B_{uv}|\le 5k/6$. 
			So, 
			\begin{align*}
			|f(V(T'))|&\ge \lceil k/2\rceil+\left(\lceil k/2\rceil-\deg^{\re{y}}(y)+r_2+1\right)\\&\qquad +\left(\deg^{\re{x}}(x)-1\right)-(s-1)-r_2\\ &\ge{5k}/{6}+1, 
		\end{align*}
			a contradiction. 
		\end{itemize}
		In each case we obtain a contradiction, which concludes the proof. 		
	\end{itemize}
	
\end{proof}

\subsection{Extending the embedding: from $B_{uv}$ to $T$}
\label{subs:extension}

In this subsection, we execute step 3 of our strategy: given an embedding of $B_{uv}$, we show that 
it can be extended to all of $T$. 
This is all that is needed to complete the proof of Lemma~\ref{lemma:third}.

\begin{proof}[Proof of Lemma~\ref{lemma:third}]
	Again, for brevity, we write $\Delta, \Delta_2$ instead of $\Delta(T), \Delta_2(T)$, and $B_{uv}$ instead of $B_{uv}(T)$, where $u,v\in V(T)$ are such that $\deg^+(u)=\Delta$ and $\deg^{\re{v}}(v)=\Delta_2\ge k/4+3$. Let $D'$ be the subgraph from Subsection~\ref{rem:cases}, 
	and let $P_{uv}$ be the (unique) path joining $u$ and $v$ in $T$. 
	
	We apply Lemma~\ref{lemma:thirdpart2} to find an embedding of $B_{uv}$ into $D'$ in cases~\ref{rem:1} and ~\ref{rem:2.1}, and a suitable embedding of $B_{uv}$ into $D$ in case~\ref{rem:2.2}. Next, we show that in both cases ~\ref{rem:1} and ~\ref{rem:2.2}, this embedding can be extended to cover all of $T$. 
	
	\begin{claim}
		\label{claim:outside-caterpillar} 
		If 
		\begin{itemize}
			\item in case \ref{rem:1}, there is an embedding of $B_{uv}$ in $D'$, and, 
			\item in case \ref{rem:2.2}, there is a suitable embedding of $B_{uv}$ in $D$.
		\end{itemize}
		Then $T$ embeds in $D$. 
	\end{claim}
	\begin{myproof}  
		Let $T'$ be a maximal subtree of $T$ such that 
		$B_{uv}\subseteq T'$ and there is an embedding of $T'$ in $D'$ in case \ref{rem:1}, or  there is a suitable embedding of $T'$ in $D$
		in case~\ref{rem:2.2}. 
		Suppose, for the sake of a contradiction, that $T'\ne T$. Let us consider $T$ as a rooted tree with $u$ as its root. 
		Let $w\in T'$ such that $\deg_{T'}^{\re{w}}(w) <\deg_{T}^{\re{w}}(w)$, so, in particular, $w\in \NL$, and thus 
		$\deg_{D'}^{\re{w}}(f(w))\ge \lceil k/2\rceil$.
		Further, the maximality of $T'$ implies that 
		\begin{equation}
			\label{eq:neighbors-ww}
			N_{D'}^{\re{w}}(f(w))\subseteq f(T'). 
		\end{equation}
		
		Let us now show that $|\LL_{T'}\setminus \LL_w|> \lceil k/2\rceil$. Indeed, since $B_{uv}\subseteq T'$, 
		every vertex in $N^+(u)\cup N^{\re{v}}(v)$ 
		that is not in the path joining $u$ and $v$, and is not a predecessor of $w$, is either a leaf of $T'$ or has a descendant  that is a leaf of $T'$. Given that
		$\Delta\ge \Delta_2\ge \lfloor{k}/{4}\rfloor+3$, we conclude that
		$|\LL_{T'}\setminus \LL_w|\ge \Delta+\Delta_2-3>\lceil k/2\rceil$, as desired. 
		Thus, there exists $x\in \LL_{T'}\setminus \LL_w$ with $f(x)\in N_{D'}^{\re{w}}(f(w))$, because $\deg_{D'}^{\re{w}}(f(w))\ge \lceil k/2\rceil$
		and $|V(T')|\le k$. 
		Fix one such vertex $x\in \LL_{T'}\setminus \LL_w$, and let $p_x$ be  its parent. 
		The maximality of $T'$ together with the fact that $x$ is a leaf of $T$ (and thus can easily be reembedded) implies that  
		\begin{equation}
			\label{eq:neighbors-y}
			N_{D'}^{\re{p_x}}(f(p_x))\subseteq f(V(T')).
		\end{equation}
		Observe that it may happen that $p_x\in \{u,v\}$;   regardless of this, we may take $z\in\{u,v\}\setminus \{p_x\}$. Then, as  $\Delta\ge \Delta_2\ge \lfloor{k}/{4}\rfloor+3$,
		\begin{equation}
			\label{eq:neighbors-z}
			|N_{D}^{\re{z}}(f(z))\cap f(V(T'))|> \lceil k/4\rceil.
		\end{equation} 
		Now,~(\ref{eq:neighbors-ww}),~(\ref{eq:neighbors-y})~and~(\ref{eq:neighbors-z}), along with~\cref{semideg_again} and the fact that $D$ is a \ksf
		digraph, imply that 
		\begin{align*}
			|f(V(T'))| &\ge  |N_{D'}^{\re{w}}(f(w))\cup N_{D'}^{\re{p_x}}(f(p_x))\cup  (N_{D}^{\re{z}}(f(z))\cap f(V(T')))|\\
			&> \lceil k/2\rceil+\lceil k/2\rceil+\lceil k/4\rceil-3(s-1)>k,
		\end{align*}
		which contradicts the hypothesis that $T'\ne T$. 
	\end{myproof}

	Because of Claim~\ref{claim:outside-caterpillar}, we assume that we are in case \ref{rem:2.1}. 
	Our aim is to embed $T$ in $D'$. 
	Since we are not in case~\ref{rem:1}, we have 
	\begin{equation}
		\label{eq:B>3k/4}
		|B_{uv}|> 3k/4.
	\end{equation}
	First suppose that $\Delta\ge \lceil 7k/12\rceil$, which means that $r=k-\Delta$.  
	Since $\Delta\ge \lceil 7k/12\rceil$ and $\Delta_2\ge \lfloor k/4\rfloor+3$, we have $|B_{uv}|\ge 5k/6+2$, which implies
	$|\NL_u|\le |\NL|\le  k/6-1$. 
	Let $T^*:=T\setminus \LL_u$, $\NL'_u:=\NL_u\setminus V(P_{uv})$ and let $D_u$ be the set of descendant vertices of some 
	$u'\in \NL'_u$. We have 
	\begin{equation}
		\label{eq:T^*}
		|T^*|=k+1-(\Delta-1)+|\NL'_u|\le  r+k/6.
	\end{equation}
	We define an embedding $f$ of $T^*$, and then  
	extend $f$ to an embedding of all of $T$. 
	We start by mapping~$u$ to a vertex $a\in D'$ with $\deg^+_{D'}(a)\ge k$. Next, we greedily embed all vertices
	of $T^*\setminus (\NL'_u\cup D_u)$, which is possible because $|T^*\setminus (\NL'_u\cup D_u)|\le r+1$ and 
	$\overline{\delta^0}(D')\ge r$.   We then embed the vertices of $\NL'_u$ into unused vertices of $D'$, which is possible because 
	$\deg_{D'}^+(a)\ge k$. Subsequently, we greedily embed the vertices of $D_u$. If this is not possible, there is
	a vertex $z\in T^*$ with $N^{\re{z}}(f(z))\subseteq \im f$. As $D'$ is \ksf and since $\deg^{\re{v}}(v)\ge \lfloor k/4\rfloor +3$, 
	there is a subset $X\subseteq f(N^{\re{v}}(v))\setminus N^{\re{z}}(f(z))$ with $|X|> k/6$. However, since $\deg^{\re{z}}(f(z))\ge r$, this would imply that
	\[|f(V(T^*))|\ge |N^{\re{z}}(f(z))|+ |X|> r+k/6,\]
	contradicting \cref{eq:T^*}. Thus, the embedding of $D_u$ is feasible. Finally, since $\deg^+(a)\ge k$, we can 
	embed the vertices of $\LL_u$ into unused vertices of $D'$, resulting in an embedding of $T$ in $D'$.

	So,   we can assume
	that $\Delta\le \lfloor 7k/12\rfloor$, and thus $r=\lceil {5k}/{12}\rceil$. 
	Then, 
	\begin{equation}
		\label{eq:mindegsec7.2}\overline{\delta^+}(D')\ge \lceil k/2\rceil\text{ and }\overline{\delta^-}(D')\ge \lceil 5k/12\rceil.	\end{equation}
	
	From this point on, our strategy is structured as follows:

	\begin{itemize}
		\item[{\bf Step 1:}] We first consider a maximal subtree $T_1$ with $B_{uv}\subseteq T_1\subseteq T$ such that there exists an embedding $f_1:V(T_1) \to V(D')$. Assuming $T_1 \neq T$ (as otherwise the proof is complete), we then select a vertex $w$ in $T_1$ with an unembedded neighbour, say $w'$. The maximality of $T_1$ ensures that $N^{\re{w}}(f_1(w)) \subseteq \im f_1$.
		
		\item[{\bf Step 2:}] Now, we proceed by selecting a specific vertex $x^*$ from $T_1$, which possesses certain properties that are essential for the next embedding step. We will embed the unembedded neighbour of $w$ in $f_1(x^*)$.  After embedding $w'$, we define a maximal subtree $T_2$ of $T$, for which there exists an embedding $f_2 : V(T_2) \to V(D')$. Notably, $T_2$ will include a vertex $z \ne  w$ with an unembedded neighbour. We will show that $z\in \{u,v\}$. Furthermore, by construction of $T_2$, we have $N^{\re{w}}(f_2(w))\subseteq \im f_2$, and by the maximality of $T_2$, we have $N^{\re{z}}(f_2(z))\subseteq \im f_2$. 
		
		\item[{\bf Step 3:}] Next, we set $y:=\{u,v\}\setminus \{z\}$, and prove that the closed neighbourhood of $y$ is contained into $T_2$, implying that $f_2(N^{\re{y}}(f_2(y))\subseteq \im f_2$.
		To summarise, the sets $f_2(N^{\re{y}}(f_2(y))$, $N^{\re{w}}(f_2(w))$ and $N^{\re{z}}(f_2(z))$ are all contained in $\im f$, with the first of these sets having at least $\lceil k/4\rceil +3$ vertices, while each of the other two has at least $\lceil 5k/12\rceil$ vertices. 
				
		\item[{\bf Step 4:}] Finally, we establish upper bounds on the pairwise intersections of the sets $f_2(N^{\re{y}}(y))$, $N^{\re{w}}(f_2(w))$, and $N^{\re{z}}(f_2(z))$, and use these to prove  that $|f_2(V(T_2))| > k$. This leads to a contradiction, as it implies $T_2 = T$, completing the proof.
	\end{itemize} 
	With this strategy established, we now proceed with the detailed execution of each step.

	Let $T_1$ be a maximal subtree of $T$ such that $B_{uv}\subseteq T_1$ and there exists an 
	embedding $f_1:V(T_1)\rightarrow V(D')$. 
	If  $T_1=T$ there is nothing to show, so suppose $T_1\ne T$. 
	Let $w\in T_1$ with $\deg_{T_1}^{\re{w}}(w) <\deg^{\re{w}}_{T}(w)$, and let 
	$w'\in N^{\re{w}}(w)\setminus V(T_1)$. Note that $w\notin \{u,v\}$, because $B_{uv}\subseteq T_1$. 
	The maximality of $T_1$ implies that 
	\begin{equation}
		\label{eq:neighbors-w}
		N^{\re{w}}(f_1(w))\subseteq f_1(V(T_1)).
	\end{equation} 
	This completes {\bf Step 1} of the outlined strategy. 
	
	Let $R'_w$ be the shortest path 
	joining $w$ and some vertex, say $p$, of $P_{uv}$. Let $R_w$ be the union of $R'_w$ and all neighbours of $w$ that belong to $T_1$. 
	Since $|B_{uv}|>{3k}/{4}$ by~\cref{eq:B>3k/4}, 	and $|R_w\cap B_{uv}|\le 2$ (it may happen that $p\in \{u,v\}$, so $2$ is a possibility), 
	it follows that $|R_w|\le |T_1\setminus B_{uv}|+2\le {k}/{4}+2$. 
	Let \[X:=\{x\in T_1\setminus R_w:f_1(x)\in N^{\re{w}}(f_1(w))\}.\] 
	Notice that $N^{\re{w}}_{T_1}(w)\subseteq R_w$, and so $X\cap N_{T_1}^{\re{w}}(w)=\emptyset$. 
	Furthermore, note that $k \geq 13$; otherwise, $s = \lceil k/12 \rceil = 1$, and $D$ contains no orientation of $K_{1,2}$ because $D$ is \ksff. This contradicts the fact that $T_1 \supseteq B_{uv}$ is already embedded into $D' \subseteq D$. 
	Then, since $\deg^{\re{w}}(f_1(w))\ge \lceil 5k/12\rceil$ by~\cref{eq:mindegsec7.2}, combined 
	with~\eqref{eq:neighbors-w} and the facts that $|R_w|\le {k}/{4}+2$ and $k\ge 13$, 
	it follows that $X\ne \emptyset$.   For each $x\in X$, let $P_x$ be the subpath of $T$ joining $w$ and~$x$. 
	
	We now proceed to execute {\bf Step 2} of our strategy. 
	Let 
	\[X^*:=\{x\in X:P_x\not\subseteq P_{x'} \text{ for all }x'\in X\setminus \{x\}\}.  \]
	The fact that $X\ne \emptyset$ implies $X^*\ne \emptyset$. 
	Choose $x^*\in X^*$ and a subtree $C$ of $T_1$ such that $x^*\notin V(C)$ and
	$R_w\cup (X\setminus \{x^*\})\subseteq C$. Among all such choices of $x^*$ and $C$, choose these such that $|C\cap B_{uv}|$ is maximum.
	Note that $u\in C$ or $v\in C$. 
	Let $T_2$ be a maximal subtree of $T$ such that $C\cup \{w'\}\subseteq T_2$ and 
	there is an embedding $f_2:V(T_2)\rightarrow V(D')$ with $f_2(t)=f_1(t)$, for each  $t\in C$, and $f_2(w')=f_1(x^*)$. 
	Observe that $N^{\re{w}}_{T_1}(w)\cup \{w'\}\subseteq T_2$ and, since $x^*\in X^*$, also  $N^{\re{w}}(f_2(w))\subseteq f_2(T_2)$. 
	On the other hand, the maximality of $T_1$ implies the existence of  a vertex $z\in T_2$ 
	such that $\deg^{\re{z}}_{T_2}(z)<\deg^{\re{z}}_{T_1}(z)$, and so $N^{\re{z}}(f_2(z))\subseteq f_2(V(T_2))$. 
	We will assume that we chose $x^*$ and $T_2$ so that, if possible, $z$ is an out-vertex (while maintaining our condition that $|C\cap B_{uv}|$ is maximised as our first priority).

	\begin{claim}
		\label{claim:second-option}
		$B_{uv}\not\subseteq T_2$.
	\end{claim}
	\begin{myproof} 
		For the sake of a contradiction suppose that $B_{uv}\subseteq T_2$. Then $z\notin \{u,v\}$. 
		By the maximality of $T_2$ we have $N^{\re{z}}(f_2(z))\subseteq f_2(V(T_2))$. 
		Let 
		\begin{equation*}
			\label{eq:neighbor-leaf}
			U:=f_2(N^+(u)), \hspace{0.5em} V:=f_2(N^{\re{v}}(v)),\hspace{0.5em} W:=N^{\re{w}}(f_2(w)) \hspace{0.5em} \text{and} \hspace{0.5em} Z:=N^{\re{z}}(f_2(z)).
		\end{equation*}
		We have $U\cup V\cup W\cup Z\subseteq f_2(V(T_2))$ 
		and $|U\cap V|\le 1$ (it may happen that the path joining $u$ and $v$ in $T$ has length $2$). 
		Note that if $\Delta+\Delta_2\ge {7k}/{12}$, then because of~\cref{eq:mindegsec7.2},
		\[|T_2|\ge |U\cup V\cup W\cup Z|\ge \Delta+\Delta_2+2\left(\lceil 5k/12\rceil\right)-5(s-1)-1>k,\]
		a contradiction since $T_2\neq T$. So we can assume  that 
		\begin{equation}
			\label{eq:sum7k12}
			\Delta+\Delta_2<{7k}/{12}. 
		\end{equation}
		
		By~\cref{eq:B>3k/4}, we have $U\cap V=\emptyset$, and further, $|T\setminus B_{uv}|\le {k}/{4}+1$, 
		which implies $|\NL_{\{u,v\}}|\le {k}/{4}+1$. Thus,  
		\begin{equation}
			\label{eq:Luv} 
			|\LL_{\{u,v\}}|=\Delta+\Delta_2-|\NL_{\{u,v\}}|> {k}/{4}	\end{equation} because $\Delta\ge \Delta_2> k/4+1$. 
		Since $|W\cup Z|\ge 2\left(\lceil 5k/12\rceil\right)-(s-1)>{3k}/{4}$, it follows that 
		\begin{equation}\label{WuZ}
			\parbox{0.85\linewidth}{\centering\textit{$(W\cup Z)\cap \LL_{\{u,v\}}\ne \emptyset$. 
			}}
		\end{equation}
		
		Suppose now $Z\cap \LL_{\{u,v\}}\ne \emptyset$, and let $y\in \{u,v\}$ such that $Z\cap \LL_y\ne \emptyset$. 
		This means that there is a leaf $y'\in N^{\re{y}}(y)$ such that $f_2(y')\in N^{\re{z}}(f_2(z))$. 
		Since $\deg^{\re{z}}_{T_2}(z)<\deg^{\re{z}}_{T_1}(z)$, the maximality of $T_2$ implies that 
		$N^{\re{y}}(f_2(y))\subseteq f_2(V(T_2))$
		(otherwise we could extend the embedding $f$ to an embedding of $T_2\cup \{z'\}$ in $D'$). 
		By~\cref{eq:mindegsec7.2}, this contradicts Lemma~\ref{lemma:k/4neighbors}, used with vertices $f_2(y)$, $f_2(z)$ and $f_2(w)$.
		Hence $Z\cap \LL_{\{u,v\}}= \emptyset$, and thus, by~\cref{WuZ}, $W\cap \LL_{\{u,v\}}\ne \emptyset$.

		Let $y\in \{u,v\}$ such that $W\cap \LL_y\ne \emptyset$. Let $y'\in W\cap \LL_y$ and let $x:=\{u,v\}\setminus \{y\}$. 
		By the choice of $x^*\in X^*$ and $C$, observe that $f_2(y)=f_1(y)$ and $f_2(y')=f_1(y')$. 
		Let us now return to the tree $T_1$ 
		and the embedding $f_1:V(T_1)\to V(D')$. Since
		$y,y'\in T_1$, and by the maximality of $T_1$, 
		it  holds that 
		$N^{\re{w}}(f_1(w))\subseteq f_1(T_1)$ and $N^{\re{y}}(f_1(y))\subseteq f_1(T_1)$. 
		So, if $y$ is an out-vertex, we also have that $w$ is an out-vertex, and then~\cref{eq:mindegsec7.2} implies that
		\[|T_1|\ge |N^{\re{y}}(f_1(y))\cup N^{\re{w}}(f_1(w))\cup f_1(N^{\re{x}}(x))|\ge k+\Delta_2-3(s-1)>k,\]
		a contradiction since $T_1\neq T$. Therefore, $y$ is an in-vertex, and hence $y=v$. As we are not in case~\ref{rem:1}, and 
		because of~\cref{eq:sum7k12}, 
		it follows  that 
		$|T\setminus B_{uv}|< \Delta-\Delta_2<{k}/{12}$. In particular, $|\NL_{\{u,v\}}|< {k}/{12}$. 
		
		We return to $T_2$ and $f_2$. 
		Since $Z\cap \LL_{u,v}= \emptyset$, it follows that $|Z\cap (U\cup V)|\le |\NL_{\{u,v\}}|< {k}/{12}$. 
		The fact that $D'$ is \ksf implies that $|W\cap (U\cup V)|<{k}/{6}$ and $|W\cap Z|<{k}/{12}$. Recall that $U\cap V=\emptyset$. Therefore,
		\begin{align*}
			|T_2|&\ge |U\cup V\cup W\cup Z| \\ &\ge |U|+|V|+|W|+|Z|- |Z\cap (U\cup V)|-|W\cap (U\cup V)|-
			|W\cap Z|\\ &\ge \Delta+\Delta_2+2\left(\lceil 5k/12\rceil\right)-{4k}/{12}
			>k
		\end{align*}
		a contradiction since $T_2\neq T$. Thus, $B_{uv}\not\subseteq T_2$, as desired. 
	\end{myproof}
	
	We note that 
	\[
	\text{$z\in P_{uv}$ and there is a vertex $z'\in B_{uv}\cap (N^{\re{z}}(z)\setminus V(T_2))$.}
	\]
	Indeed, $\deg^{\re{t}}_{T_2}(t)=\deg^{\re{t}}_{T}(t)$ for each $t\in  T_2\setminus \{w,z\}$, as otherwise  $N^{\re{t}}(f_2(t))\subseteq f_2(T_2)$
	by the maximality of $T_2$, but it is easy to see that this contradicts Lemma~\ref{lemma:k/4neighbors} applied to vertices $f_2(t)$, $f_2(w)$ and $f_2(z)$.  
	So by Claim~\ref{claim:second-option}, 
	$z\in P_{uv}$, and there is a vertex $z'$ with the above properties.

	If $X\cap (T_1\setminus B_{uv})\ne \emptyset$, we could have chosen $x^*\in X^*\cap (T_1\setminus B_{uv})$, 	and so we would have $B_{uv}\subseteq C\subseteq T_2$, contradicting Claim~\ref{claim:second-option}. 
	Thus, we may assume that $X\cap (T_1\setminus B_{uv})=\emptyset$, or equivalently, that $X\subseteq B_{uv}$. 
	Recall that $C$ contains at least one of $u$ and $v$, and $C\subseteq T_2$. 
	Let $y\in \{u,v\}\setminus \{z\}$, so $N^{\re{y}}[y]\subseteq V(T_2)$. Set 
	\[\parbox{0.85\linewidth}{\centering\text{$Y:=f_2(N^{\re{y}}(y))$, \, $W:=N^{\re{w}}(f_2(w))$\, and \, $Z:=N^{\re{z}}(f_2(z))$.}}\] 
	Then
	\[|f_2(V(T_2))|\ge |Y\cup W\cup Z|\ge \Delta_2+2\left(\lceil 5k/12\rceil\right)-3(s-1)>{5k}/{6}.\]
	In particular, $|T_2|>{3k}/{4}$ and therefore
	\begin{equation}\label{zuv}
		\text{$z\in \{u,v\}$. }
	\end{equation} 
	Note that the vertices $y$ and $z$ are the ones identified according to {\bf Step 3} of our strategy. Now, we claim that 
	\begin{equation}\label{zinvertex}
		\text{$z$ is an in-vertex.}
	\end{equation}
	Suppose otherwise. Then, by~\cref{eq:mindegsec7.2},  $\deg^+(f_2(z))\ge \lceil k/2\rceil$. We will argue now that $\deg^{\re{w}}(f_2(w))+\deg^+(f_2(z))\ge k$. Indeed, 
	if $w$ is an in-vertex, this follows from~\ref{cor:2} of Lemma~\ref{cor:subdigraph}, 
	and if $w$ is an out-vertex, then it follows from~\cref{eq:mindegsec7.2}. Thus,   
	\[|f_2(V(T_2))|\ge |Y\cup W\cup Z|\ge \Delta_2+k-3(s-1)>k,\]
	a contradiction since $T_2\neq T$. This proves~\cref{zinvertex}.
	
	We are now ready to proceed with {\bf Step 4}, the final step in our strategy. The primary goal here is to establish a lower bound on the pairwise intersections of the sets $Y$, $W$, and $Z$, which in turn will allow us to establish a lower bound on $|\im f_2|$. 
	
	By~\cref{zuv} and~\cref{zinvertex}, $z=v$ and $y=u$. Further, since we chose $x^*$ and $T_2$ so that $z$ was an out-vertex if possible, 
	we know that $X\cap N^+(u)=\emptyset$,  so 
	\begin{equation}\label{second-to-last-eq}
		\text{$W\cap Y=\emptyset$.}
	\end{equation}

	We claim that 
	\begin{equation}\label{last-eq}
		\text{$|Y\cap Z|\le 1$.}
	\end{equation}
	Indeed,
	suppose otherwise, and let $y'\in N^+(y)$ not on the path joining $u$ and $v$, with $f_2(y')\in Y\cap Z$. As $f_2(w)$ and $f_2(z)$ each send at least $\lceil 5k/12\rceil$ arcs to $\im f_2$ (a set of at most $k$ vertices), 
	Lemma~\ref{lemma:k/4neighbors} implies that $|N_{D'}^+(f_2(u))\cap \im f_2|<5k/12$. Therefore, by~\cref{eq:mindegsec7.2}, there exists a vertex $a\in N_{D'}^+(f_2(u))\setminus f_2(V(T_2))$. Also,
	recall that $W\subseteq f_2(B_{uv}\cup R_w)$, because $X\subseteq B_{uv}$ and due to the definition of $T_2$ and $f_2$. 
	Let $C'$ be the  connected component of $T_2\setminus \{y'\}$ that contains $w$, and let $T_3$ be a maximal subtree of $T$
	such that $C'\cup \{z',y'\}\subseteq T_3$ and there is an embedding, say $f_3$, of $T_3$ in $D'$ with $f_3(t)=f_2(t)$, for each
	$t\in C'$, $f_3(z')=f_2(y')$ and $f_3(y')=a$. Note that 
	$|T_3\cap B_{uv}|\ge |T_2\cap B_{uv}|+1$, which contradicts the choice of $C\subseteq T_2$. This proves~\cref{last-eq}.
	
	By~\cref{second-to-last-eq} and~\cref{last-eq},  and given that $|W\cap Z|\le s-1$ and $\Delta_2\ge \lfloor k/4\rfloor +3$, we have 
	\[|f_2(V(T_2))|\ge |Y\cup W\cup Z|\ge \Delta_2+2\left(\lceil 5k/12\rceil\right)-(s-1)-1>k,\]
	a contradiction that concludes the proof. 
\end{proof}

\small\bibliographystyle{abbrvnat} \bibliography{biblio}

\begin{thebibliography}{14}
\providecommand{\natexlab}[1]{#1}
\providecommand{\url}[1]{\texttt{#1}}
\expandafter\ifx\csname urlstyle\endcsname\relax
  \providecommand{\doi}[1]{doi: #1}\else
  \providecommand{\doi}{doi: \begingroup \urlstyle{rm}\Url}\fi

\bibitem[Addario-Berry et~al.(2013)Addario-Berry, Havet, Sales, Reed, and
  Thomass{\'e}]{ahlrt}
L.~Addario-Berry, F.~Havet, C.~L. Sales, B.~Reed, and S.~Thomass{\'e}.
\newblock Oriented trees in digraphs.
\newblock \emph{Discrete Math.}, 313\penalty0 (8):\penalty0 967--974, 2013.

\bibitem[Balasubramanian and Dobson(2007)]{Dob01}
S.~Balasubramanian and E.~Dobson.
\newblock Constructing trees in graphs with no {$K\sb {2,s}$}.
\newblock \emph{J. Graph Theory}, 56\penalty0 (4):\penalty0 301--310, 2007.

\bibitem[Besomi et~al.(2021)Besomi, Pavez-Sign{\'e}, and Stein]{BPS3}
G.~Besomi, M.~Pavez-Sign{\'e}, and M.~Stein.
\newblock On the {E}rd{\H{o}}s-{S}\'os conjecture for trees with bounded
  degree.
\newblock \emph{Combin. Probab. Comput.}, 30\penalty0 (5):\penalty0 741--761,
  2021.

\bibitem[Burr(1980)]{burr}
S.~A. Burr.
\newblock Subtrees of directed graphs and hypergraphs.
\newblock In \emph{Proceedings of the Eleventh South- eastern Conference on
  Combinatorics, Graph Theory and Combinatorics (Florida Atlantic Univ., Boca
  Raton, Fla.}, volume I,28, pages 227--239, 1980.

\bibitem[Chen et~al.(2024)Chen, Hou, and Zhou]{chen2024}
B.~Chen, X.~Hou, and H.~Zhou.
\newblock Long antipaths and anticycles in oriented graphs.
\newblock \emph{arXiv preprint arXiv:2401.05205}, 2024.

\bibitem[Erd\H{o}s(1964)]{Erdos64}
P.~Erd\H{o}s.
\newblock Extremal problems in graph theory.
\newblock \emph{Theory of graphs and its applications, Proc. Sympos.
  Smolenice\/}, pages 29--36, 1964.

\bibitem[Graham(1970)]{graham}
R.~L. Graham.
\newblock On subtrees of directed graphs with no path of length exceeding one.
\newblock \emph{Canad. Math. Bull.}, 13\penalty0 (3):\penalty0 329--332, 1970.

\bibitem[Kalai(2017)]{blog}
G.~Kalai.
\newblock Micha {P}erles' {G}eometric {P}roof of the {E}rd{\H o}s--{S}\'os
  {C}onjecture for {C}aterpillars.
\newblock
  \url{https://gilkalai.wordpress.com/2017/08/29/micha-perles-geometric-proof-of-the-erdos-sos-conjecture-for-caterpillars/},
  2017.
\newblock Accessed: March 2024.

\bibitem[Klimo{\v{s}}ov{\'a} and Stein(2023)]{antipaths}
T.~Klimo{\v{s}}ov{\'a} and M.~Stein.
\newblock Antipaths in oriented graphs.
\newblock \emph{Discrete Math.}, 346\penalty0 (9):\penalty0 113515, 2023.

\bibitem[Moser and Pach(1993)]{moser}
W.~Moser and J.~Pach.
\newblock Recent developments in combinatorial geometry.
\newblock \emph{New Trends in Discrete and Computational Geometry}, pages
  281--302, 1993.

\bibitem[Rozho\v{n}(2019)]{rozhon}
V.~Rozho\v{n}.
\newblock A local approach to the {{E}rd\H{o}s}--{{S}\'os} conjecture.
\newblock \emph{SIAM J. Discrete Math.}, 33\penalty0 (2):\penalty0 643--664,
  2019.

\bibitem[Stein(2021)]{survey}
M.~Stein.
\newblock Tree containment and degree conditions, in {A}. {R}aigoroskii and
  {M}. {R}assias, editors.
\newblock \emph{Discrete Math. Appl.}, pages 459--486, 2021.

\bibitem[Stein(2024)]{maya-digraphs}
M.~Stein.
\newblock Oriented trees and paths in digraphs.
\newblock In F.~Fischer and R.~Johnson, editors, \emph{Surveys in Combinatorics
  2024}, London Mathematical Society Lecture Note Series. Cambridge University
  Press, 2024.

\bibitem[Stein and Z{\'a}rate-Guer{\'e}n(2024)]{camila}
M.~Stein and C.~Z{\'a}rate-Guer{\'e}n.
\newblock Antidirected subgraphs of oriented graphs.
\newblock \emph{Combin. Probab. Comput.}, pages 1--21, 2024.

\end{thebibliography}

\end{document}